\documentclass[11pt]{amsart}
\usepackage{amssymb}
\usepackage{verbatim}
\usepackage{hyperref}
\usepackage{color}

\begin{document}

\newtheorem{theorem}{Theorem}
\newtheorem{proposition}[theorem]{Proposition}
\newtheorem{conjecture}[theorem]{Conjecture}
\def\theconjecture{\unskip}
\newtheorem{corollary}[theorem]{Corollary}
\newtheorem{lemma}[theorem]{Lemma}
\newtheorem{sublemma}[theorem]{Sublemma}
\newtheorem{fact}[theorem]{Fact}
\newtheorem{observation}[theorem]{Observation}
\theoremstyle{definition}
\newtheorem{definition}{Definition}
\newtheorem{notation}[definition]{Notation}
\newtheorem{remark}[definition]{Remark}
\newtheorem{question}[definition]{Question}
\newtheorem{questions}[definition]{Questions}
\newtheorem{example}[definition]{Example}
\newtheorem{problem}[definition]{Problem}
\newtheorem{exercise}[definition]{Exercise}



\def\reals{{\mathbb R}}
\def\torus{{\mathbb T}}
\def\heis{{\mathbb H}}
\def\integers{{\mathbb Z}}
\def\rationals{{\mathbb Q}}
\def\naturals{{\mathbb N}}
\def\complex{{\mathbb C}\/}
\def\distance{\operatorname{distance}\,}
\def\support{\operatorname{support}\,}
\def\dist{\operatorname{dist}}
\def\Dist{\operatorname{Distance}}
\def\Span{\operatorname{span}\,}
\def\degree{\operatorname{degree}\,}
\def\kernel{\operatorname{kernel}\,}
\def\dim{\operatorname{dim}\,}
\def\codim{\operatorname{codim}}
\def\trace{\operatorname{trace\,}}
\def\Span{\operatorname{span}\,}
\def\dimension{\operatorname{dimension}\,}
\def\codimension{\operatorname{codimension}\,}
\def\nullspace{\scriptk}
\def\kernel{\operatorname{Ker}}
\def\ZZ{ {\mathbb Z} }
\def\p{\partial}
\def\rp{{ ^{-1} }}
\def\Re{\operatorname{Re\,} }
\def\Im{\operatorname{Im\,} }
\def\ov{\overline}
\def\eps{\varepsilon}
\def\lt{L^2}
\def\diver{\operatorname{div}}
\def\curl{\operatorname{curl}}
\def\etta{\eta}
\newcommand{\norm}[1]{ \|  #1 \|}
\def\expect{\mathbb E}
\def\bull{$\bullet$\ }
\def\det{\operatorname{det}}
\def\Det{\operatorname{Det}}
\def\multiR{\mathbf R}
\def\bestA{\mathbf A}
\def\Apq{\mathbf A_{p,q}}
\def\Apqr{\mathbf A_{p,q,r}}
\def\bestC{\mathbf C}
\def\bestAqd{\mathbf A_{q,d}}
\def\bestBqd{\mathbf B_{q,d}}
\def\bestB{\mathbf B}
\def\Apq{\mathbf A_{p,q}}
\def\Apqr{\mathbf A_{p,q,r}}
\def\rank{\mathbf r}
\def\diameter{\operatorname{diameter}}
\def\bp{\mathbf p}
\def\bff{\mathbf f}
\def\bx{\mathbf x}
\def\bu{\mathbf u}
\def\bg{\mathbf g}
\def\essd{\operatorname{essential\ diameter}}

\def\symdif{\,\Delta\,}
\def\frakE{{\mathfrak E}}
\def\frakI{{\mathfrak I}}
\def\defe{\dist(E,\frakE)}
\def\defi{\dist(E,\frakI)}

\def\mab{\max(|A|,|B|)}
\def\t2{\tfrac12}
\def\tatb{tA+(1-t)B}
\def\divergence{\operatorname{div}}
\def\steiner{\bigstar}

\newcommand{\abr}[1]{ \langle  #1 \rangle}

\newcommand{\Norm}[1]{ \Big\|  #1 \Big\| }
\newcommand{\set}[1]{ \left\{ #1 \right\} }
\def\one{{\mathbf 1}}
\newcommand{\modulo}[2]{[#1]_{#2}}

\def\scriptf{{\mathcal F}}
\def\scriptq{{\mathcal Q}}
\def\scriptg{{\mathcal G}}
\def\scriptm{{\mathcal M}}
\def\scriptb{{\mathcal B}}
\def\scriptc{{\mathcal C}}
\def\scriptt{{\mathcal T}}
\def\scripti{{\mathcal I}}
\def\scripte{{\mathcal E}}
\def\scriptv{{\mathcal V}}
\def\scriptw{{\mathcal W}}
\def\scriptu{{\mathcal U}}
\def\scriptS{{\mathcal S}}
\def\scripta{{\mathcal A}}
\def\scriptr{{\mathcal R}}
\def\scripto{{\mathcal O}}
\def\scripth{{\mathcal H}}
\def\scriptd{{\mathcal D}}
\def\scriptl{{\mathcal L}}
\def\scriptn{{\mathcal N}}
\def\scriptp{{\mathcal P}}
\def\scriptk{{\mathcal K}}
\def\scriptP{{\mathcal P}}
\def\scriptj{{\mathcal J}}
\def\scriptz{{\mathcal Z}}
\def\scripts{{\mathcal S}}
\def\frakv{{\mathfrak V}}
\def\frakG{{\mathfrak G}}
\def\aff{\operatorname{Aff}}
\def\frakB{{\mathfrak B}}
\def\frakC{{\mathfrak C}}
\def\aA{{\mathfrak A}}

\def\aff{\operatorname{Aff}}
\def\bestC{{\mathbf C}}

\def\bE{{\mathbf E}}
\def\bF{{\mathbf F}}
\def\bG{{\mathbf G}}
\def\Psharp{P^\sharp}
\def\bA{{\mathbf A}}
\def\bI{{\mathbf I}}
\def\bEstar{{\mathbf E}^\star}
\def\bAstar{{\mathbf A}^\star}
\def\be{{\mathbf e}}
\def\bv{{\mathbf v}}
\def\bw{{\mathbf w}}
\def\br{{\mathbf r}}
\def\Star{\star}
\def\unitQ{{\mathbf Q}}
\def\Gl{\operatorname{Gl}}

\def\Gjo{G_{\text{$j$,o}}}
\def\Gje{G_{\text{$j$,e}}}

\def\ball{\mathbb B}

\def\barrier{\medskip\hrule\hrule\medskip}

\def\bb{\mathbb B}
\def\br{\mathbf{r}}
\def\bt{\mathbf{t}}
\def\sstar{{\dagger\star}}
\def\Star{{\bullet}}
\def\aff{\operatorname{Aff}}
\def\gl{\operatorname{Gl}}
\def\baff{\operatorname{\bf Aff}}
\newcommand{\mbf}[1]{\mathbf #1}
\def\T{{\mathcal T}}

\def\defb{|E\symdif \bb|}

\author{Michael Christ}

\address{
        Michael Christ\\
        Department of Mathematics\\
        University of California \\
        Berkeley, CA 94720-3840, USA}
\email{mchrist@berkeley.edu}
\thanks{Research supported in part by NSF grant DMS-1363324}

\date{January 29, 2017. Edited April 29, 2017.} 

\title{A sharpened Riesz-Sobolev inequality}

\begin{abstract}
The Riesz-Sobolev inequality provides an upper bound, in integral form, for the convolution of indicator functions of subsets of Euclidean space. 
We formulate and prove a sharper form of the inequality. This can be equivalently phrased as a stability result,
quantifying an inverse theorem of Burchard that characterizes cases of equality.
\end{abstract}

\maketitle



\section{Introduction}
\subsection{The Riesz-Sobolev inequality}
Let $\bE=(E_1,E_2,E_3)$ be an ordered triple of Lebesgue measurable subsets of $\reals^d$
with finite Lebesgue measures.
The Riesz-Sobolev inequality \cite{riesz},\cite{sobolev} states that
\begin{equation} \label{eq:RS1}
\int_{E_3} \one_{E_1}*\one_{E_2}
\le \int_{E_3^\star} \one_{E_1^\star}*\one_{E_2^\star}
\end{equation}
where $E^\star$ denotes the closed ball, centered at the origin, that satisfies $|E^\star|=|E|$,
$*$ denotes convolution of functions,
and $\one_E$ denotes the indicator function $\one_E(x)=1$ if $x\in E$ and $=0$ if $x\notin E$.
This can be read both as an upper bound for 
$\int_{E_3} \one_{E_1}*\one_{E_2}$ as $E_j$ vary over all sets of prescribed measures,
and as a statement that this upper bound is attained by $\bEstar = (E_1^\star,E_2^\star,E_3^\star)$.

Burchard \cite{burchard} characterized those triples $\bE$ that realize equality  in \eqref{eq:RS1}.
Such a characterization must take into account two features of the inequality, namely affine invariance
and the concept of admissibility.
Affine invariance holds in the sense that if $\psi:\reals^d\to\reals^d$ is a measure-preserving
linear transformation and $\bv = (v_1,v_2,v_3)\in (\reals^d)^3$ satisfies $v_3=v_1+v_2$
then the sets $\tilde E_j = \psi(E_j)+v_j$ satisfy
$\int_{\tilde E_3} \one_{\tilde E_1}*\one_{\tilde E_2}
= \int_{E_3} \one_{E_1}*\one_{E_2}$ and $|\tilde E_j|=|E_j|$.
For arbitrary $\psi\in\Gl(d)$, 
$\int_{\tilde E_3} \one_{\tilde E_1}*\one_{\tilde E_2}
= |\det(\psi)|^2 \int_{E_3} \one_{E_1}*\one_{E_2}$. 

An ordered triple $\br = (r_1,r_2,r_3)$ of positive real numbers is said to be admissible if
$r_k\le r_i+r_j$ for all permutations $(i,j,k)$ of $(1,2,3)$,
and to be strictly admissible if $r_k < r_i+r_j$ for all permutations. 
An ordered triple $\bE$ of measurable subsets of $\reals^d$ 
is said to be admissible (respectively strictly admissible)
if $(|E_j|^{1/d}: 1\le j\le 3)$ is admissible (respectively strictly admissible).
Burchard's theorem states that if $\bE$ is strictly admissible and 
realizes equality in the Riesz-Sobolev inequality, then there exist a measure-preserving
linear transformation $\psi$ and $\bv\in (\reals^d)^3$ satisfying $v_3=v_1+v_2$
such that $E_j = \psi(E_j^\star)+v_j$ for all $j\in\{1,2,3\}$.
In particular, the sets $E_j$ are mutually homothetic ellipsoids.
Here, and throughout this paper, two sets are regarded as equivalent if their symmetric difference
is a Lebesgue null set. 

In the borderline admissible but not strictly admissible case, equality holds if and only if
the sets $E_j$ are (equivalent to) suitably translated mutually homothetic convex sets \cite{burchard}. 
This is equivalent to the well-known characterization of equality in the Brunn-Minkowski inequality.
This borderline case will not be discusssed in the present paper.

In order to state our main result we need the following notion of distance
from $\bE$ to the best approximating ordered triple of compatibly translated homothetic ellipsoids
of appropriate Lebesgue measures.
\begin{definition}
Let $\bE=(E_1,E_2,E_3)$ and $\bF=(F_1,F_2,F_3)$  be ordered triples of Lebesgue measurable
subsets of $\reals^d$ with $|E_j|,|F_j|<\infty$ for each index $j$.
The distance from $\bE$ to the orbit of $\bF$ is
\begin{equation}
\Dist(\bE,\scripto(\bF)) = \inf_{\psi,\bv}
\max_{j\in\{1,2,3\}} |E_j\symdif (\psi(F_j)+v_j)|
\end{equation}
where the infimum is taken over all $\bv=(v_1,v_2,v_3)\in(\reals^d)^3$
satisfying $v_3 = v_1+v_2$
and over all Lebesgue measure--preserving invertible linear automorphisms $\psi$ of $\reals^d$.
\end{definition}

We will be especially interested in $\Dist(\bE,\scripto(\bEstar))$.
It is elementary that this quantity vanishes if and only if
there exist $\psi,\bv$, with $\psi$ measure-preserving and $\bv$ satisfying $v_3=v_1+v_2$, 
such that $E_j = \psi(E_j^\star)+v_j$ for all $j\in\{1,2,3\}$.

We also require a quantitative concept of strict admissibility.
\begin{definition}
Let $\rho>0$. 
An ordered triple $\br$ of positive real numbers is $\rho$--strictly admissible
if $r_k\le (1-\rho)(r_i+r_j)$ for all permutations $(i,j,k)$ of $(1,2,3)$,
and $\min(r_1,r_2,r_3) \ge\rho\max(r_1,r_2,r_3)$.

An ordered triple $\bE$ of Lebesgue measurable subsets of $\reals^d$ with positive,
finite Lebesgue measures is $\rho$--strictly admissible if $(|E_j|^{1/d}: 1\le j\le 3)$
is a $\rho$--strictly admissible triple of positive real numbers.
\end{definition}

Our main result is:
\begin{theorem} \label{thm:RSsharpened}
For each $d\ge 1$ and each $\rho>0$ there exists $c>0$ such that
for each $\rho$--strictly admissible ordered triple $\bE$
of Lebesgue measurable subsets of $\reals^d$,
\begin{equation} \label{eq:RSsharpened}
\int_{E_3} \one_{E_1}*\one_{E_2}
\le \int_{E_3^\star} \one_{E_1^\star}*\one_{E_2^\star}
- c \Dist(\bE,\scripto(\bEstar))^2.  \end{equation}
\end{theorem}

The exponent $2$ is optimal.
This bound does not hold in the borderline admissible case.

Theorem~\ref{thm:RSsharpened} can also be read as a characterization of those triples $\bE$
that nearly extremize the Riesz-Sobolev functional. 

\begin{theorem}  \label{thm:inverse}
For each $d\ge 1$ and each $\rho>0$ there exists $c>0$ 
such that for any $\delta>0$ and any
$\rho$--strictly admissible ordered triple $\bE$
of Lebesgue measurable subsets of $\reals^d$,
if 
\[\int_{E_3} \one_{E_1}*\one_{E_2} \ge 
(1-\delta) \int_{E_3^\star} \one_{E_1^\star}*\one_{E_2^\star}\]
then
\[ \Dist(\bE,\scripto(\bEstar))\le C\delta^{1/2} \max_j |E_j|.\]
\end{theorem}

The exponent $\tfrac12$ is optimal.
Burchard's theorem, in the strictly admissible range, is a corollary.

A sharper version, treating the dependence on $\rho$ more quantitatively,
was established for $d=1$ in \cite{christRS3}. This improved on a weaker version 
\cite{christrieszsobolev},
whose main hypothesis was that $\int_{E_{3,i}} \one_{E_1}*\one_{E_2}$
should be nearly maximal for sets $E_{3,1}$ and $E_{3,2}$ with $|E_{3,2}|/|E_{3,1}|$
nearly equal to an odd integer.
This weaker result was still sufficient to serve as the key ingredient in a characterization
of near-maximizers for
Young's convolution inequality in \cite{christyoungest}.\footnote{\cite{christyoungest}
was subsequently revised to incorporate the simpler formulation in \cite{christRS3}.}
The proof was quite different from the method
developed below, relying on a result from
additive combinatorics concerning sets whose sumsets have nearly minimal size, adapted
from the discrete case to the continuum context. For $d>1$, a less quantative form 
\[ \int_{E_3} \one_{E_1}*\one_{E_2}
\le \int_{E_3^\star} \one_{E_1^\star}*\one_{E_2^\star}
\ -\   \Theta\Big(\frac{\Dist(\bE,\scripto(\bEstar))}{\max_j|E_j|} \Big)\cdot \max_j |E_j|^2 \]
of \eqref{eq:RSsharpened}
was established in \cite{christRShigher}, with an unspecified function $\Theta$ that vanishes only at $0$.

\subsection{A variant inequality}
The following inequality for subsets of $\reals^d$ is closely related to the Riesz-Sobolev inequality.
For $d=1$ it is discussed in \cite{christRS3}. A proof is included below.

\begin{theorem} \label{thm:RSvariant}
For any $d\ge 1$ and any Lebesgue measurable sets $E_j\subset\reals^d$ with finite Lebesgue measures,
for any $\tau>0$,
\begin{equation}\label{eq:variantRS}
\int_{\reals^d} \min(\one_{E_1}*\one_{E_2},\tau) 
\ge \int_{\reals^d} \min(\one_{E_1^\star}*\one_{E_2^\star},\tau). 
\end{equation}
\end{theorem}

A sharpened form, parallel to Theorem~\ref{thm:RSsharpened}, is as follows.
Define 
$\Dist((A,B),\scripto(A^\star,B^\star))$ to be the infimum, over all Lebesgue measure-preserving linear
linear transformations $\psi$ of $\reals^d$ and over all $(u,v)\in (\reals^d)^2$, 
of $\max\big(|A\symdif (\psi(A^\star)+u)|,\,|B\symdif (\psi(B^\star)+v)|\big)$. 

\begin{theorem} \label{thm:variantinverse}
For any $d\ge 1$, any compact set $\Lambda\subset(0,1)$, and any $\rho>0$
there exists $c>0$ such that for any  Lebesgue measurable sets $E_j\subset\reals^d$ with finite,
positive Lebesgue measures
satisfying $\min(|E_1|,|E_2|)\ge\rho \max(|E_1|,|E_2|)$
and any $\tau\in\reals^+$ such that $\tau/\min(|E_1|,|E_2|)\in\Lambda$,
\begin{equation} 
\int_{\reals^d} \min(\one_{E_1}*\one_{E_2},\tau) 
\ge \int_{\reals^d} \min(\one_{E_1^\star}*\one_{E_2^\star},\tau) 
+ c\Dist((E_1,E_2),\scripto(E_1^\star,E_2^\star))^2.
\end{equation}
\end{theorem}

In particular, for $0<\tau<\min(|E_1|,|E_2|)$, equality holds in \eqref{eq:variantRS}
only if $(E_1,E_2)$ is a pair of homothetic ellipsoids.
For $d=1$ a slightly more quantitative result is proved in \cite{christRS3}.

A formally sharper variant holds.
Given $E_1,E_2,\tau$,
define \begin{align*}
S_\tau&=\{x\in\reals^d: (\one_{E_1}*\one_{E_2})(x)>\tau\}
\\
S_\tau^\sharp&=\{x\in\reals^d: (\one_{E_1^\star}*\one_{E_2^\star})(x)>\tau\}.
\end{align*}
\begin{theorem} \label{thm:variantinverse2}
For any $d\ge 1$, any compact set $\Lambda\subset(0,1)$, and any $\rho>0$
there exists $c>0$ such that for any  Lebesgue measurable sets $E_j\subset\reals^d$ with finite,
positive Lebesgue measures
satisfying $\min(|E_1|,|E_2|)\ge\rho \max(|E_1|,|E_2|)$
and any $\tau\in\reals^+$ such that $\tau/\min(|E_1|,|E_2|)\in\Lambda$,
\begin{equation} 
\begin{aligned}
\int_{\reals^d} \min(\one_{E_1}*\one_{E_2},\tau) 
&\ge \int_{\reals^d} \min(\one_{E_1^\star}*\one_{E_2^\star},\tau) 
\\  &
\ \ + c\Dist((E_1,E_2),\scripto(E_1^\star,E_2^\star))^2
+ c(|S_\tau|-|S^\sharp_\tau|)^2.
\end{aligned}
\end{equation}
\end{theorem}
A corresponding improvement of Theorem~\ref{thm:RSsharpened} holds.
See \cite{christRS3} for the case $d=1$.


The author is indebted to Guy David for an insightful question,
and to Almut Burchard for a useful conversation and for providing a reference.

\section{Notation and reformulation}

$\one_E$ denotes the indicator function of a set $E$, and $|E|$ denotes its Lebesgue measure. 
All functions in this paper are real-valued. $\langle f,g\rangle = \int fg$; the integral is understood
to be taken over $\reals^d$ with respect to Lebesgue measure unless the contrary is explicitly indicated.
$A\symdif B = (A\setminus B) \cup (B\setminus A)$ denotes the symmetric difference between two sets.
Since $|A\symdif B| = \norm{\one_A-\one_B}_{L^1}$, one has the triangle inequality
$|A\symdif C| \le |A\symdif B| + |B\symdif C|$ for any three measurable sets.

It will be convenient to reformulate the Riesz-Sobolev inequality in more symmetric form.
Let $\lambda$ be the natural Lebesgue measure on $\Sigma = \{\bx=(x_1,x_2,x_3)\in (\reals^d)^3: x_1+x_2+x_3=0\}$;
\begin{equation}
\int_\Sigma f(\bx)\,d\lambda(\bx) = \int_{\reals^d\times\reals^d} f(x_1,x_2,-x_1-x_2)\,dx_1\,dx_2.
\end{equation}
Define
\begin{equation}
\scriptt(\bE) = \int_\Sigma \prod_{j=1}^3 \one_{E_j}(x_j)\,d\lambda(\bx).
\end{equation}
This is equal to $\int_{-E_3} \one_{E_1}*\one_{E_2}$, where $-E_3=\{-x: x\in E_3\}$.
Since $(-E)^\star \equiv E^\star$, the Riesz-Sobolev inequality is 
equivalent to \[\scriptt(\bE)\le\scriptt(\bEstar)\ \text{ for all } \bE.\]

\begin{definition}
The distance from $\bE$ to the orbit of $\bF$ is
\begin{equation}
\dist(\bE,\scripto(\bF)) = \inf_{\psi,\bv}
\max_{j\in\{1,2,3\}} |E_j\symdif (\psi(F_j)+v_j)|
\end{equation}
where the infimum is taken over all $\bv=(v_1,v_2,v_3)\in(\reals^d)^3$
satisfying $ v_1+v_2+v_3=0$
and over all Lebesgue measure--preserving invertible linear automorphisms $\psi$ of $\reals^d$.
\end{definition}

Our main result can equivalently be formulated as follows. 
\begin{theorem} \label{thm:RSsharpenedscriptt}
For each $d\ge 1$ and each $\rho>0$ there exists $c>0$ such that
for each $\rho$--strictly admissible ordered triple $\bE$
of Lebesgue measurable subsets of $\reals^d$,
\begin{equation} \label{eq:RSsharpened2}
\scriptt(\bE)\le \scriptt(\bEstar) - c \Dist(\bE,\scripto(\bEstar))^2.  \end{equation}
\end{theorem}

The Steiner symmetrization $E^\dagger$ of a Lebesgue measurable set $E\subset\reals^d$,
satisfying $|E|<\infty$, is defined as follows.
Regard $\reals^d$ as $\reals^{d-1}\times\reals^1$ with coordinates $(x',t)$.
Define the vertical slices $E_{x'}=\{t\in\reals: (x',t)\in E\}$.
Denote by $O(d)$ the group of all orthogonal linear transformations of $\reals^d$.
\begin{definition}
\begin{equation} E^\dagger= \{(x',t): |t|\le \tfrac12 |E_{x'}|\}.\end{equation}
For $\scripto\in O(d)$, 
\begin{equation} E^\dagger_\scripto= \scripto^{-1}\big[(\scripto(E))^\dagger\big].
\end{equation}
\end{definition}
If $|E_{x'}|<\infty$
for every $x'$ then $E^\dagger=\{(x',t): t \in (E_{x'})^\star\}$,
where ${E_{x'}}^\star$ denotes the symmetrization of $E_{x'}\subset \reals^1$.
$E^\dagger_\scripto$ is the Steiner symmetrization of 
$E$ in the direction $\scripto^{-1}(0,0,\dots,0,1)\in\reals^d$.

\section{A flow of sets}

For general sets,
the following result is perhaps best described as folklore. It was known long ago to Burchard \cite{burchardoral},
and a version of it is mentioned in her 2009 lecture notes \cite{burchardlecturenotes}. It
appears in a recent work of Carillo, Hittmeir, Volzone, and Yao \cite{CHVY}. 
While it is not essential to our analysis, it does simplify one step,
and deserves to be more widely known.

\begin{proposition}\label{prop:RSflow}
There exists a flow $(t,E)\mapsto E(t)$ of equivalence classes of Lebesgue measurable subsets
of $\reals^1$ with finite measures, defined for $t\in[0,1]$,
having the following properties for all sets $E$:
\begin{enumerate}
\item
$E(0)=E$ and $E(1) = E^\star$.
\item 
Measure preserving:
$|E(t)| = |E|$ for all $t\in[0,1]$.
\item
Continuity:
$|E(s)\symdif E(t)|\to 0$ as $s\to t$.
\item
Inclusion monotonicity:
If $E\subset\tilde E$ then $E(s)\subset \tilde E(s)$ for all $s\in[0,1]$.
\item
Contractivity: 
$|E_1(t)\symdif E_2(t)|\le |E_1\symdif E_2|$ for all sets $E_1,E_2$.
\item
Independence of past history:
If $0\le s\le t\le 1$ then
$E(t)$ depends only on $E(s),s,t$.
Moreover, $E(t) = (E(s))(\tau)$
where $1-\tau = \frac{1-t}{1-s}$.
\item 
Functional monotonicity and continuity:
The function $t\mapsto \scriptt(E_1(t),E_2(t),E_3(t))$ is continuous and nondecreasing on $[0,1]$.
\end{enumerate}
\end{proposition}

All of these statements are to be interpreted in terms of equivalence classes
of measurable sets, with $E$ equivalent to $E'$ whenever $|E\symdif E'|=0$.
In the case in which the initial set $E$ is a finite union of pairwise disjoint closed intervals,
this flow is a well known device \cite{BLL}, \cite{liebloss}. 
From its construction it is clear that for sets that are finite unions of closed intervals this
flow satisfies
$|E(t)|=|E|$, and if $E_1\subset E_2$ then $E_1(t)\subset E_2(t)$. 
We now sketch a proof of the extension of the flow to arbitrary sets.

\begin{lemma}
For $j=1,2,3$ let $E_j$ be a finite union of pairwise disjoint closed bounded intervals.
Let $t\mapsto E_j(t)$ be as defined in \cite{BLL}, \cite{liebloss}. Then
$|E_1(t)\symdif E_2(t)| \le |E_1\symdif E_2|$.
\end{lemma}

\begin{proof}
Consider the flows of $A=E_1\cap E_2$ and $B=E_1\cup E_2$. Both of these
sets are finite unions of closed bounded intervals, so their flows are defined.
Since $A\subset B$, $A(t)\subset B(t)$ for all $t$.
Therefore
\begin{multline*} |A(t)\symdif B(t)| = |B(t)\setminus A(t)|
= |B(t)|-|A(t)|
\\
= |B|-|A|
= |(E_1\cup E_2)\setminus (E_1\cap E_2)|
= |E_1\symdif E_2|.\end{multline*}
Since $A\subset E_j\subset B$, $A(t)\subset E_j(t)\subset B(t)$ for all $t$.
Therefore 
\[ A(t)\subset E_1(t)\cap E_2(t)\subset E_1(t)\cup E_2(t)\subset B(t)\]
and consequently
\[ E_1(t)\symdif E_2(t) \subset A(t)\symdif B(t).\]
\end{proof}

To define the flow for a general Lebesgue measurable set $E\subset\reals^1$
satisfying $|E|<\infty$,
consider any approximating sequence of sets $E_n$ satisfying
$|E_n\symdif E|\to 0$ as $n\to\infty$.
Since $|E_m\symdif E_n| \le |E_m\symdif E| + |E\symdif E_n|$,
for each $t$ we have
$\lim_{m,n\to\infty} |E_m(t)\symdif E_n(t)|=0$.
Therefore (since $L^1$ is complete) 
there exists a set $\tilde E(t)$ such that $\lim_{n\to\infty} |E_n(t)\symdif \tilde E(t)|= 0$.
Moreover, this set $\tilde E(t)$ is independent of the choice of approximating
sequence $(E_n)$. Define the flow by setting $E(t) = \tilde E(t)$. 

It is clear that if $A$ is a finite union of bounded closed intervals
then $t\mapsto A(t)$ is continuous at $t=0$ in the sense that $|A(t)\symdif A|\to 0$
as $t\to 0$. 
From this and the contraction property $|A(t)\symdif B(t)|\le |A\symdif B|$
it follows immediately that $t\mapsto E(t)$ is continuous at $t=0$
for any set $E$ of finite Lebesgue measure. 
Continuity of $t\mapsto E(t)$ at an arbitrary $s\in[0,1]$
follows from this together with the independence of past history.

It is straightforward to verify the other conclusions of Proposition~\ref{prop:RSflow}. 
\qed

An auxiliary property of this flow is useful: $|E(t)\symdif E^\star|$ is a nonincreasing function
of $t$. This holds, by inspection, for finite unions of intervals, and follows for general sets
by continuity of $t\mapsto E(t)$.

We record in passing a smoothing property of this flow, 
which is not used in the proofs of our main results.
It is established in \S\ref{section:flowtointervals}.

\begin{proposition} \label{prop:flowtointervals}
Let $E\subset\reals^1$ be a Lebesgue measurable set with finite measure.
For each $t>0$, $E(t)$ equals a union of intervals.
\end{proposition}
That is, there exists a countable family of intervals $I_n$ 
such that $|E(t) \symdif \cup_n I_n|=0$.

A less canonical, but still useful, higher-dimensional analogue of Proposition~\ref{prop:RSflow}
can be constructed by combining this flow with iterated Steiner symmetrization.
The next lemma, used in this construction, is proved in \cite{BLL} and in \cite{liebloss}.
\begin{lemma}\label{lemma:iteratedSteiner}
Let $d\ge 1$. Let $\bE$ be  an ordered triple of bounded Lebesgue measurable subsets of $\reals^d$,
each with positive, finite measure. 
There exists a sequence $\scripto_n\in O(d)$
such that the sequence of iterated Steiner symmetrizations defined recursively by
$E_{0,j}=E_j$ and $E_{n,j} = (E_{n-1,j})^\dagger_{\scripto_n}$ satisfies
\[ \lim_{n\to\infty} |E_{n,j}\symdif E_j^\star| = 0\ \text{ for each $j\in\{1,2,3\}$.}\]
\end{lemma}

In higher dimensions there exists a natural flow $t\mapsto E(t)$ satisfying $E(0)=E$ and
$E(1)=E^\dagger$, the Steiner symmetrization of $E$. 
For each $x'\in\reals^d$,
define $E_j(t)\subset\reals^d$ by setting the fiber $\{x_d\in\reals: (x',x_d)\in E_j(t)\}$ 
equal to the flow defined above, at time $t$,
of the fiber $\{x_d\in\reals: (x',x_d)\in E_j\}$.
We call this the Steiner flow in the proof of Proposition~\ref{prop:RSflowd>1}.

A property of Steiner flow is that 
\begin{equation}\label{eq:daggermonotonicity} |E(t)\symdif E^\dagger| \text{ is a nonincreasing function of $t$.}
\end{equation}
This property, for $d>1$, is an immediate consequence of the case $d=1$.

The next result asserts the existence of a flow with corresponding properties
for subsets of $\reals^d$, for arbitrary $d>1$.
This higher-dimensional analogue is not canonical; its construction involves
certain choices; but it is sufficient for our purpose.

\begin{proposition}\label{prop:RSflowd>1}
Let $d\ge 1$.
For $j\in\{1,2,3\}$ let
$E_j\subset\reals^d$ be a bounded Lebesgue measurable set. 
There exist mappings $[0,1]\owns t\mapsto E_j(t)$
of equivalence classes of Lebesgue measurable subsets
of $\reals^d$, with the following properties: 
\begin{enumerate}
\item
$E_j(0)=E_j$ and $E_j(1)=E_j^\star$.
\item 
$|E_j(t)| = |E_j|$ for all $t\in[0,1]$.
\item
$|E_j(s)\symdif E_j(t)|\to 0$ as $s\to t$.
\item 
The function $t\mapsto \scriptt(\bE(t))$ is continuous and nondecreasing on $[0,1]$.
\end{enumerate}
\end{proposition}

\begin{proof}
Let $\bE$ be given.
Let $(\scripto_n: n\in\naturals)$ be as in Lemma~\ref{lemma:iteratedSteiner}.
Define the flow $t\mapsto \bE(t)$ for $t\in[0,\tfrac12]$
to be the Steiner flow of $E$,
conjugated with the rotation $\scripto_1$,
with the time rescaled so that $\bE_1$ is reached at $t=\tfrac12$ rather than at $t=1$.
Next, define the flow for $t\in[\tfrac12,\tfrac34]$ so that $\bE(\tfrac34)=\bE_2$,
by using the same construction, conjugated by $\scripto_2$.
Use the time interval $[1-2^{-k},1-2^{-k-1}]$ in the same way
to make a continuous deformation from $\bE_{k}$ to $\bE_{k+1}$
for each $k\in\naturals$.
Define $E(1)=E^\star$.
The flow thus defined is clearly continuous on $[0,1)$, and 
$\scriptt(\bE(t))$ is a nondecreasing function of $t$.

From \eqref{eq:daggermonotonicity} and the property that
$|E_{n,j}\symdif E_j^\star|\to 0$ as $n\to\infty$
it follows that
$|E_j(t)\symdif E_j^\star|\to 0$ as $t\to 1$ from below.
Thus defining $E_j(1)=E_j^\star$,  
$t\mapsto \scriptt(\bE(t))$ becomes a
continuous function on the closed interval $[0,1]$.
\end{proof}

\section{Reduction to small measure perturbations}

In order to establish Theorem~\ref{thm:RSsharpenedscriptt}, it suffices to prove it in the following
perturbative regime.

\begin{proposition}\label{prop:1}
For each $d\ge 1$ and each $\rho>0$ there exist $\delta_0>0$ and $c>0$ such that
the inequality 
\begin{equation*} 
\scriptt(\bE)\le \scriptt(\bEstar) - c \Dist(\bE,\scripto(\bEstar))^2  \end{equation*}
holds
for each $\rho$--strictly admissible ordered triple $\bE$
of Lebesgue measurable subsets of $\reals^d$ satisfying
\begin{equation}
\dist(\bE,\scripto(\bEstar))\le\delta_0 \max_{1\le j\le 3} |E_j|.
\end{equation}
\end{proposition}

In this section, we show how Theorem~\ref{thm:RSsharpenedscriptt} 
is a consequence of Proposition~\ref{prop:1}, which in turn will be proved below. 
We may suppose without loss of generality that $\max_j|E_j|=1$. For if $rE_j=\{rx: x\in E_j\}$
and $r\bE = (rE_1,rE_2,rE_3)$
then 
\[ \frac{\scriptt(r\bE)}{\max_j|rE_j|^2}
= \frac{\scriptt(\bE)}{\max_j|E_j|^2}\]
and $\dist\big( (rE_1,rE_2,rE_3),\,\scripto\big((rE_1)^\star,(rE_2)^\star,(rE_3)^\star)\big) \big) 
= r^d \dist(\bE,\scripto(\bEstar))$.
Likewise, $\bE$ is $\rho$--strictly admissible if and only if $\bE$ is so.
Thus the conclusion of Theorem~\ref{thm:RSsharpenedscriptt}
holds for $\bE$ with a given constant $c$,
if and only if it holds for $r\bE$, with the same constant $c$.
Choosing $r =\max_j |E_j|^{1/d}$ reduces matters to the case in which $\max_j|E_j|=1$.

Let $\delta_0>0$ and suppose \eqref{eq:RSsharpened2} to be known for all $\rho$--strictly admissible $\bE$
satisfying $\max_j|E_j|=1$ and $\dist(\bE,\scripto(\bEstar))\le\delta_0$.
Let $\bE$ be a $\rho$--strictly admissible ordered triple of bounded sets satisfying  $\max_j|E_j|=1$
and $\dist(\bE,\scripto(\bEstar))>\delta_0\max_j|E_j| = \delta_0$.
Consider $\bE(t) = (E_1(t),E_2(t),E_3(t))$, where $[0,1)\owns t\mapsto \bE(t)$ is the flow of
Proposition~\ref{prop:RSflowd>1}.

The function $[0,1)\owns t\mapsto \dist(\bE(t),\scripto(\bEstar))$ is continuous,
and it tends to $0$ as $t\to 1$ since $|E_j(t)\symdif E_j^\star|\to 0$.
Therefore there exists a smallest $t_0\in(0,1]$ for which $\dist(E(t_0),\scripto(\bEstar))=\delta_0$.
By monotonicity of $\scriptt$ under the threefold flow,
\begin{equation}
\scriptt(\bE) \le \scriptt(\bE(t_0)) \le \scriptt(\bEstar) - c\dist(\bE(t_0),\scripto(\bEstar))^2
=\scriptt(\bEstar)-c\delta_0^2.
\end{equation}
The maximum possible value of $\dist(\bA,\scripto(\bAstar))$, as $\bA$ ranges over all ordered triples
of sets satisfying $\max_j|A_j|=1$, is equal to $2$. Therefore 
$\scriptt(\bEstar)-c\delta_0^2\le \scriptt(\bEstar)-c'\dist(\bE,\scripto(\bEstar))^2$, 
yielding the desired inequality by transitivity.

\section{Reduction to perturbations near the boundary}

Let $\be=(e_1,e_2,e_3)\in (\reals^+)^3$ be given, and suppose that
$(e_j^{1/d}: 1\le j\le 3)$ is $\rho$--strictly admissible.
Let $E_j\subset\reals^d$ be Lebesgue measurable sets satisfying $|E_j|=e_j$. 
Define
\begin{equation} \delta =\dist(\bE,\scripto(\bEstar)).\end{equation}

Suppose that $\max_j|E_j|=1$, and that $\delta$ is small.
Choose $\psi,\bv$ so that $\tilde E_j = \psi(E_j)+v_j$ satisfy 
\begin{equation} \max_j |\tilde E_j \symdif E_j^\star| \le 2\dist(\bE,\scripto(\bEstar)),
\end{equation}
and replace $\bE$ by $(\tilde E_j: 1\le j\le 3)$,
as we may do without affecting the inequality in question.

Let $B_j=E_j^\star$, and let $r_j>0$ be the radius of $B_j$.
Define functions $f_j$ by
\begin{equation} \one_{E_j} = \one_{E_j^\star} + f_j = \one_{B_j}+f_j.\end{equation}
Thus $f_j$ takes values in $\{-1,0,1\}$, $\int f_j=0$,
and the essential support of $f_j$ has measure $\le 2|E_j\symdif E_j^\star| \le 4\delta$.

One has
\[ \scriptt(\bE) = \scriptt(\one_{B_1}+f_1,\,\one_{B_2}+f_2,\,\one_{B_3}+f_3)
= \scriptt(\bE^\star) + \sum_{k=1}^3 \langle K_k,f_k\rangle + O(\delta^2)\]
where the kernels $K_k$ are defined for $k\in\{1,2,3\}$ by \begin{equation} K_k = \one_{B_i}*\one_{B_j}\end{equation}
with the notation $\{1,2,3\} = \{i,j,k\}$.
Each of the functions $K_k$ is radial, nonnegative, and nonincreasing along each ray emanating
from the origin.
Moreover, if $r_i\le r_j$ then $\nabla K_k(x)\ne 0$
whenever $r_j-r_i<|x|<r_j+r_i$.
Strict admissibility of $\bE$ is equivalent to the assertion that 
$r_j-r_i < r_k < r_i+r_j$ for all permutations $(i,j,k)$ of $(1,2,3)$.
Therefore $\nabla K_k(x)$ is nonzero when $|x| = r_k$.
Define $\gamma_k$ by
\begin{equation}\label{eq:gammakdefn} |\nabla K_k(x)|=\gamma_k \text{ when } |x|=r_k.\end{equation}
We will abuse notation mildly by writing $K_k(s)$ for $K_k(x)$ where $|x|=s$.

Since $|E_k\setminus B_k| = |B_k\setminus E_k|$,
$\int f_k=0$ and consequently
\begin{align*} \langle K_k,f_k\rangle 
& = \int (K_k(x)-K_k(r_k))f_k(x)\,dx
\\&
= - \int |K_k(x)-K_k(r_k)|\cdot |f_k(x)|\,dx
\end{align*}
since
the functions $x\mapsto K_k(x)-K_k(r_k)$  and $-f_k$ are both
nonnegative on $B_k$ and nonpositive on its complement. 

Let $\lambda$ be a large positive constant, which is to be independent
of $\delta$ and is to be chosen below.
If $\lambda\delta$ is bounded above by a small constant depending only on $\be$
then it follows from the nonvanishing of $\nabla K_k$ in a neighborhood
of $|x|=r_k$ that 
\begin{align*}
\langle K_k,f_k\rangle 
&\le -c\lambda\delta \int_{|\,|x|-r_k\,|\ge\lambda\delta} |f_k(x)|\,dx
\\&
= -c\lambda\delta \big|\,\{x\in E_k\symdif B_k: |\,|x|-r_k\,|\ge\lambda\delta\}\,\big|.
\end{align*}

We aim to reduce to the case in which $E_k\symdif B_k$
is entirely contained in $\{x: \big|\, |x|-r_k\,\big| \le\lambda\delta\}$
for each index $k$. To accomplish this,
for each index $j\in\{1,2,3\}$ choose a set $E_j^\ddagger$ so that 
\begin{gather*}
|E_j^\ddagger| = |E_j|,
\\
\text{$E_j\symdif B_j$ is the disjoint union of $E_j^\ddagger\symdif B_j$
and $E_j\symdif E_j^\ddagger$}
\\
\{x\in E_j\symdif B_j: \big|\,|x|- r_j\,\big|>\lambda\delta\}
\subset E_j^\ddagger\symdif E_j
\\
|E_j^\ddagger\symdif E_j|\le 2
\big|\{x\in E_j\symdif B_j: \big|\,|x|-r_j\,\big|>\lambda\delta\}\big|.
\end{gather*}
To construct such a set, 
define $S =  \{x\in E_j\symdif B_j: \big|\,|x|-r_j\,\big| \le \lambda\delta\}$,
$S_+ = S\cap (E_j\setminus B_j)$, and $S_- = S\cap (B_j\setminus E_j)$.
If $|S_+|\ge |S_-|$ 
then choose $\tilde S_+\subset S_+$ to be a measurable set satisfying
$|\tilde S_+| = |S_-|$,
and
define $E_j^\ddagger$ by 
\begin{equation} E_j^\ddagger = (B_j\cup \tilde S_+)\setminus S_-.\end{equation}
If on the other hand
$|S_+| < |S_-|$ 
then choose $\tilde S_-\subset S_-$ to be a measurable set satisfying
$|\tilde S_-| = |S_+|$,
and
define $E_j^\ddagger$ by 
\begin{equation} E_j^\ddagger = (B_j\cup S_+)\setminus \tilde S_-.\end{equation}
Then $E_j\symdif B_j$ is the disjoint union of $E_j^\ddagger\symdif B_j$ and $E_j\symdif E_j^\ddagger$,  so
\begin{equation}\label{eq:sumrule} 
|E_j\symdif B_j| = |E_j\symdif E_j^\ddagger| + |E_j^\ddagger\symdif B_j|.\end{equation}
Moreover, $E_j^\ddagger\symdif B_j\subset\{x: \big|\,|x|-r_j\,\big| \le \lambda\delta\}$.

Set $f_j^\ddagger = \one_{E_j^\ddagger}-\one_{B_j}$
and expand \[\one_{E_j} = \one_{B_j} + f_j^\ddagger + (f_j-f_j^\ddagger)
= \one_{E_j^\ddagger} + \tilde f_j\]
where $\tilde f_j = f_j-f_j^\ddagger =\one_{E_j}-\one_{E_j^\ddagger}$.
Thus $\tilde f_j$ takes values in $\{-1,0,1\}$ and has essential support equal to $E_j\symdif E_j^\ddagger$.
Write $\bE^\ddagger = (E_1^\ddagger,E_2^\ddagger,E_3^\ddagger)$.

\begin{lemma} \label{lemma:removeddagger}
Let $d\ge 1$ and $\rho>0$.
There exist $\lambda<\infty$, and $\delta_0,c>0$, with the following property. 
Let $\bE$ be a $\rho$--strictly admissible ordered triple of subsets 
of $\reals^d$ satisfying $\max_j|E_j|=1$
and $\max_j|E_j\symdif E_j^\star|\le\delta_0$.
Define $\bE^\ddagger$ as above. Then
\begin{equation} \label{eq:daggered}
\scriptt(\bE) \le \scriptt(\bEstar) -c \lambda \sum_{i=1}^3 |E_i\symdif E_i^\star| 
\,\cdot\, \sum_{j=1}^3 |E_j\symdif E_j^\ddagger|.
\end{equation} \end{lemma}

In the following proof, $c$ denotes a small strictly positive constant, whose value is
permitted to change from one occurrence to the next.
\begin{proof}
Set $\delta = \max_{i=1}^3 |E_i\symdif E_i^\star|\le\delta_0$.
Write $\one_{E_k} = \one_{E_k^\star}+f_k^\ddagger + \tilde f_k$ for each index,
and expand $\scriptt(\bE)$ accordingly to obtain $27$ terms. 
Eight of those terms do not involve the functions $\tilde f_j$;
these eight recombine to give $\scriptt(\bE^\ddagger)$.
Three terms are of the form $\langle K_k,\tilde f_k\rangle$;
their sum is less than or equal to 
$-c\lambda \delta \sum_k |E_k\symdif E_k^\ddagger|$, as discussed above.
The remaining terms involve either two or more $\tilde f_j$,
or one $\tilde f_j$ and at least one $f_k^\ddagger$. 
By the elementary inequality
\begin{equation} \label{eq:trivialbound}
\scriptt(E_1,E_2,E_3) \le \min_{i\ne j\in\{1,2,3\}} |E_i|\cdot|E_j|,
\end{equation}
each of these terms is
\[O(\max_j|E_j\symdif E_j^\star|\cdot \max_k |E_k^\ddagger\symdif E_k|)
= O(\delta \max_k |E_k^\ddagger\symdif E_k|),\]
uniformly in $\lambda$. 
Thus
\begin{equation}\label{eq:tobeabsorbed}
\scriptt(\bE)
\le \scriptt(\bE^\ddagger)
-c\lambda \delta \sum_j |E_j\symdif E_j^\ddagger|
+ O(\delta \sum_j |E_j\symdif E_j^\ddagger|)
\end{equation}
where the constant factor implicit in the $O$ notation is independent of $\lambda$.

Choose $\lambda$ to be a sufficiently large constant to ensure that
the remainder term in \eqref{eq:tobeabsorbed} can be absorbed, yielding 
\begin{equation} \label{eq:predaggered}
\scriptt(\bE) \le \scriptt(\bE^\ddagger) -c \lambda \sum_{i=1}^3 |E_i\symdif E_i^\star| 
\,\cdot\, \sum_{j=1}^3 |E_j\symdif E_j^\ddagger|
\end{equation}
with a smaller but still positive value of $c$.
Since $|E_j^\ddagger| = |E_j|$, $(E_j^\ddagger)^\star = E_j^\star$.  
Therefore by the Riesz-Sobolev inequality, $\scriptt(\bE^\ddagger)\le \scriptt(\bEstar)$.
Inserting this into \eqref{eq:predaggered} yields \eqref{eq:daggered}.
\end{proof}

If $\max_{1\le j\le 3} |E_j\symdif E_j^\ddagger|\ge \tfrac1{10} \max_{1\le j\le 3} |E_j\symdif E_j^\star|$
we conclude immediately from Lemma~\ref{lemma:removeddagger} that 
\[\scriptt(\bE) \le \scriptt(\bEstar) -c \sum_{j=1}^3 |E_j\symdif E_j^\star|^2
\le \scriptt(\bEstar) -c \dist(\bE,\scripto(\bEstar))^2.\]

There remains the main case, in which
$\max_j |E_j\symdif E_j^\ddagger|\le \tfrac1{10} \max_j |E_j\symdif E_j^\star|$.
In this case the nonpositive term
$-c \lambda \sum_{i=1}^3 |E_i\symdif E_i^\star| \,\cdot\, \sum_{j=1}^3 |E_j\symdif E_j^\ddagger|$
in \eqref{eq:daggered} may not be useful. However, \eqref{eq:predaggered} still gives
\begin{equation} \scriptt(\bE) \le \scriptt(\bE^\ddagger) \end{equation} 
and it suffices to prove that $\scriptt(\bE^\ddagger)
\le \scriptt(\bEstar) - c \dist(\bE,\scripto(\bE^\ddagger))^2$.
Indeed,
\[ \dist(\bE^\ddagger,\scripto(\bEstar)) \ge \dist(\bE,\scripto(\bEstar))-\max_j |E_j\symdif E_j^\ddagger|
\ge \tfrac12 \dist(\bE,\scripto(\bEstar)).\]

\section{Reduction to the boundar(ies)} \label{section:bdyreduction}
We have shown that $\bE$ may be replaced by $\bE^\ddagger$.
Thus our assumptions henceforth are that $\bE$ is $\rho$--strictly admissible,
that $\max_j|E_j|=1$, 
that $\dist(\bE,\scripto(\bEstar))\le\delta_0$,
that $\delta = \max_{i=1}^3 |E_i\symdif E_i^\star|$ satisfies 
$\delta \le 4 \dist(\bE,\scripto(\bEstar))$, 
and that 
\begin{equation} \label{eq:nearboundary} 
E_j\symdif B_j\subset\{x: \big|\,|x|-r_j\,\big|\le\lambda\delta.\}\end{equation}
We aim to prove that for any $\lambda,\rho\in (0,\infty)$
there exists $\delta_0>0$ such that
$\scriptt(\bE)$ satisfies the desired inequality
whenever all of these assumptions are satisfied.

We continue to write $B_j = E_j^\star$ and to denote by $r_j$ the radius of $B_j$.
The kernels $K_j$ and positive constants $\gamma_j$ are as defined above.
It is elementary that $K_k$ is twice continuously differentiable in a neighborhood of
the boundary of $B_k$.
We write $\be = (e_j: 1\le j\le 3) = (|E_j|: 1\le j\le 3)$.

Continue to represent $\one_{E_j} = \one_{B_j} + f_j$. Each $f_j$ is supported in
$\{x: \big|\,|x|-r_j\,\big|\le\lambda\delta\}$,
and satisfies $\int f_j=0$ and $\norm{f_j}_{L^\infty}\le 1$.
Refine this representation by defining functions
$f_j^\pm$, taking values in $\{0,1\}$, by $\one_{E_j\setminus B_j} =  f_j^+$
and  $\one_{B_j\setminus E_j} =  f_j^-$,
so that $f_j = f_j^+-f_j^-$.
Introduce polar coordinates $(r,\theta)$ in $\reals^d$
and define functions $F_j^\pm\in \lt(S^{d-1})$ by
\[ F_j^\pm(\theta) = \int_{\reals^+} f_j^\pm(t\theta)\,t^{d-1}\,dt\]
where $\theta$ is regarded as a unit vector, so that $t\theta$ is
the point with polar coordinates $(t,\theta)$.
Define $F_j\in L^2(S^{d-1})$ by $F_j=F_j^+-F_j^-$.

Under the hypothesis that $E_j\symdif B_j$ is contained in a small neighborhood
of the boundary of $B_j$,
\begin{equation}
|E_j\symdif B_j|^2\asymp 
\norm{F_j^+}_{L^2(S^{d-1})}^2
+ \norm{F_j^-}_{L^2(S^{d-1})}^2,
\end{equation}
where $u\asymp v$ means that $u\le Cv$ and $v\le Cu$ with a positive, finite constant $C$
that depends on $d,\rho$ but not otherwise on $\bE$.
Thus it suffices to establish an upper bound of the form
\[\scriptt(\bE)\le\scriptt(\bEstar) - c\sum_{j=1}^3 \big( \norm{F_j^+}_{L^2(S^{d-1})}^2
+ \norm{F_j^-}_{L^2(S^{d-1})}^2\big).\]


Denote by $\sigma$ the rotation-invariant measure on $S^{d-1}$, normalized so that
Lebesgue measure in $\reals^d$ is equal to $r^{d-1}\,dr\,d\sigma(\theta)$ in polar coordinates.
For each $k\in\{1,2,3\}$ define the quadratic form $\scriptq_k$ on $L^2(S^{d-1})$ by  
\begin{equation} \scriptq_k(F,G) = \iint_{S^{d-1}\times S^{d-1}} F(x)G(y)
\one_{|r_i x+r_j y|\le r_k}\,d\sigma(x)\,d\sigma(y).
\end{equation}
Define also
\begin{equation} \scriptq(F_1,F_2,F_3) 
= \scriptq_1(F_2,F_3) + \scriptq_2(F_3,F_1) + \scriptq_3(F_1,F_2).
\end{equation}

Whenever $|x|=|y|=1$, $|r_i x+r_j y|^2 = r_i^2+r_j^2+2r_ir_jx\cdot y$.
Therefore $\scriptq_k$ is symmetric, and
the compact linear operator $T:L^2(S^{d-1})\to L^2(S^{d-1})$
defined by
$TF(x) = \int_{S^{d-1}} F(y)\one_{|r_i x+r_j y|\le r_k}\,d\sigma(y)$
is selfadjoint, and commutes with rotations. 

The goal of this section is the following second order expansion, in 
whose formulation $L^2$ denotes $L^2(S^{d-1},\sigma)$.
\begin{proposition}\label{prop:RSperturbative}
Under the hypotheses introduced at the beginning of \S\ref{section:bdyreduction},
\begin{equation} \label{eq:SHexpansion}
\scriptt(\bE) \le \scriptt(\bEstar) 
-\tfrac12 \sum_{k=1}^3 \gamma_k r_k^{-(d-1)} \big(\norm{F_k^+}_{L^2}^2 +\norm{F_k^-}_{L^2}^2\big)
+ \scriptq(F_1,F_2,F_3) + O(\delta^3).  \end{equation} \end{proposition}

To begin the proof, substitute $\one_{B_j}+f_j$ for $\one_{E_j}$ for each index,
and expand $\scriptt(\bE)$ as a sum of the resulting eight terms. 
The main term is $\scriptt(B_1,B_2,B_3) = \scriptt(\bEstar)$.
There are three other types of terms, which are analyzed in the next three lemmas.

\begin{lemma} $|\scriptt(f_1,f_2,f_3)| = O(\delta^3)$.  \end{lemma}

\begin{proof}
If $x_1\in\reals^d$ satisfies $\big|\,|x_1|-r_1\,\big|\le C\delta$
then the $\sigma$--measure
of the set\footnote{For $d=1$ this set is empty.}
of all  $x_2\in\reals^d$ satisfying both $\big|\,|x_2|-r_2\,\big|\le C\delta$
and $\big|\,|x_1+x_2|-r_3\,\big|\le C\delta$
is $O_\rho(\delta^2)$ under the hypothesis of $\rho$--strict admissibility,
provided that $C\delta$ is sufficiently small relative to $\rho$.
\end{proof}

\begin{lemma} For each index $j\in\{1,2,3\}$,
\begin{equation}\label{eq:Fjterm} \langle K_j,f_j\rangle
\le -\tfrac12 \gamma_j r_j^{1-d} (\norm{F_+}_{L^2}^2 + \norm{F_-}_{L^2}^2) + O(\delta^3).\end{equation}
\end{lemma}

\begin{proof}
It is elementary that for each index $k\in\{1,2,3\}$,
$K_k$ is twice continuously differentiable in a neighborhood of the support of $f_k$. Therefore 
since $f_k(t\theta)=0$ unless $|t-r_j|\le C\delta$,
\begin{align*}
\langle K_k,\,f_k\rangle 
&= \int_{S^{d-1}}\int_0^\infty K_k(t\theta) f_k(t\theta) t^{d-1}\,dt\,d\sigma(\theta)
\\& 
= \int_{S^{d-1}}\int_0^\infty 
\big(
K_k(r_k) -\gamma_k (t-r_k)+O(\delta^2)
\big)\,
f_k(t\theta) t^{d-1}\,dt\,d\sigma(\theta).
\end{align*}
The expression 
$K_k(r_k) - \gamma_k (t-r_k)+O(\delta^2)$ leads to three terms.
The first of these is $K_k(r_k)\int_{\reals^d} f_k=0$.
The third is $O(\delta^2)\norm{f_k}_{L^1} = O(\delta^3)$.
The second is 
$ -\gamma_k \int_{S^{d-1}}\int_0^\infty (t-r_k)
(f_k^+-f_k^-)(t\theta) t^{d-1}\,dt\,d\sigma(\theta)$.
The integrand is nonnegative, since  
$f_k^+-f_k^-$ has the same sign as $t-r_k$.

Let us momentarily imagine
that  $F_k^+$ is given, and that $f_k^+$ varies among those functions supported in $\{t\theta: t\ge r_k\}$,
taking values in $\{0,1\}$, satisfying 
$\int_0^\infty f_k^+(t\theta) t^{d-1}\,dt = F_k^+(\theta)$.
Among all such functions $f_k$, $\int_{\reals^+} f_k^+(t\theta) t^{d-1}\,dt$
is minimized if $t\mapsto f_k(t\theta)$ is the indicator function of an interval $[r_k,r_k+h(\theta)]$
where $h(\theta)\ge 0$ is defined by the relation $\int_{r_k}^{r_k+h(\theta)} t^{d-1}\,dt = F_k^+(\theta)$.
That is, $(r_k+h(\theta))^d-r_k^d = dF_k^+(\theta)$.
Therefore, since $h(\theta)=O(\delta)$, 
\begin{equation}\label{eq:hformula} h(\theta) = r_k^{-(d-1)} F_k^+(\theta) + O(\delta F_k^+(\theta)).\end{equation}
This gives 
\begin{equation} \int_{\reals^+} (t-r_k) f_k^+(t\theta) t^{d-1}\,dt
\ge \int_{r_k}^{r_k+h(\theta)} (t-r_k)  t^{d-1}\,dt
= \tfrac12 r_k^{d-1}h(\theta)^2 - O(h(\theta)^3) \end{equation}
where the constants implicit in the $O(\cdot)$ notation depend only on $d,r_k,\lambda,\delta_0$.
The remainder term $O(h(\theta)^3)$ is $O(\delta^3)$ according to \eqref{eq:hformula},
since $\norm{F_k^+}_{L^\infty} = O(\delta)$.
Moreover, $\tfrac12 r_k^{d-1}h(\theta)^2 = \tfrac12 r_k^{-(d-1)}F_k^+(\theta)^2$,
also by \eqref{eq:hformula}.

The same analysis may be applied to $f_k^-,F_k^-$, yielding \eqref{eq:Fjterm}.
\end{proof}

\begin{lemma}
\[ \scriptt(f_i,f_j,\one_{B_k}) = \scriptq_k(F_i,F_j) + O(\delta^3).\]
\end{lemma}

That is,
\begin{multline} \label{eq:thatis}
\iint_{\reals^d\times\reals^d} f_i(x)f_j(y)\one_{B_k}(x+y)\,dx\,dy
\\ = \iint_{S^{d-1}\times S^{d-1}} F_i(x)F_j(y)\one_{|r_i x+ r_j y|\le r_k}\,d\sigma(x)\,d\sigma(y)
+  O(\delta^3).
\end{multline}

\begin{proof}
For any $\theta_i,\theta_j\in S^{d-1}$,
\[ \iint_{\reals^+\times\reals^+} f_i(t_i\theta_i)f_j(t_j\theta_j)t_i^{d-1}t_j^{d-1}
\one_{|t_i\theta_i+t_j\theta_j|\le r_k}(t_i,t_j)\,dt_i\,dt_j
= F_i(\theta_i)F_j(\theta_j)\one_{|r_i\theta_i+r_j\theta|\le r_k} \]
unless $\big|\,|r_i\theta_i+r_j\theta_j|-r_k\,\big|\le C\delta$.
The $\sigma\times\sigma$ measure of the set of all ordered pairs $(\theta_i,\theta_j)\in S^{d-1}\times S^{d-1}$
satisfying 
$\big|\,|r_i\theta_i+r_j\theta_j|-r_k\,\big|\le C\delta$
is $O(\delta)$ by the $\rho$--strict admissibility hypothesis.
Since $\int |f_i(t_i\theta_i)| t_i^{d-1}\,dt_i=O(\delta)$, $F_i = O(\delta)$,
and the same bounds hold with $i$ replaced by $j$,
the total contribution of all such exceptional pairs $(\theta_i,\theta_j)$ to
either side of \eqref{eq:thatis} is $O(\delta^3)$.
\end{proof}

Combining the last three lemmas establishes Proposition~\ref{prop:RSperturbative}.

\section{Spectral analysis}

A natural question, in light of what has been shown thus far, is what is the value of the optimal constant
$A$ in the inequality
\begin{equation} \label{eq:A}
\scriptq(F_1,F_2,F_3) \le A\sum_{k=1}^3 \gamma_k r_k^{1-d} \norm{F_k}_{\lt}^2 \ \text{ $\forall\, F_j\in L^2(S^{d-1})$
satisfying $\textstyle\int_{S^{d-1}} F_j\,d\sigma=0$.}\end{equation}
This is potentially relevant because of the inequality
\begin{equation}  
\norm{F_k}_{L^2}^2,
= \norm{F_k}_{L^2}^2 +2\langle F_k^+,F_k^-\rangle
 \le \norm{F_k^+}_{L^2}^2 +\norm{F_k^-}_{L^2}^2,
\end{equation}
which is valid since $F_k^\pm$ are nonnegative and $F_k = F_k^+-F_k^-$.

The optimal constant $A$ in \eqref{eq:A} cannot be strictly less than $\tfrac12$. For if \eqref{eq:A} were
to hold for some $A<\tfrac12$ then it would be a direct consequence of
the foregoing analysis that for any $\rho$--strictly admissible $\bE$ satisfying
$\max_j|E_j|=1$ with $\max_j|E_j\symdif E_j^\star|$ sufficiently small
and $E_j\symdif E_j^\star\subset\{x: \big|\,|x|-r_j\,\big|\le C\max_i|E_i\symdif E_i^\star|\}$ for all $j$,
$\scriptt(\bE)\le \scriptt(\bEstar) - c\max_j |E_j\symdif E_j^\star|^2$.
But this conclusion is false; by virtue of affine invariance of the functional $\scriptt$, 
it fails for every admissible ordered triple $\bE$ 
of homothetic ellipsoids with centers $\bv$ with $\sum_j v_j=0$. 
Thus the analysis must be refined, to exploit the full strength of the
assumption that $\max_j|E_j\symdif E_j^\star|$
has the same order of magnitude as $\dist(\bE,\scripto(\bEstar))$.

For this purpose, we recast $\scriptq$ in terms of spherical harmonics.
For each nonnegative integer $n$, denote by $\scripth_n\subset L^2(S^{d-1})$ the finite-dimensional
subspace of all spherical harmonics of degree $n$. Then $L^2(S^{d-1})= \oplus_{n=0}^\infty \scripth_n$.
Denote by $\pi_n$ the orthogonal projection from $L^2(S^{d-1})$ onto $\scripth_n$.
This decomposition diagonalizes each of the quadratic forms $\scriptq_k$,
in the sense that there exist compact selfadjoint operators $T_k$ on $L^2(S^{d-1})$
such that 
$\scriptq_k(F,G) \equiv \langle T_k(F),G\rangle$,
$T_k: \scripth_n\to\scripth_n$ for all $n$, and $T_k$ agrees with a scalar multiple 
$\lambda(n,r_1,r_2,r_3)$ on $\scripth_n$.

Because $|E_j|=|E_j^\star|$, $\int_{\reals^d} f_j=0$ for each index $j$
and consequently $\int_{S^{d-1}}F_j\,d\sigma=0$; that is, $\pi_0(F_j)=0$.
Therefore for each $d\ge 2$,
\begin{equation}\label{eq:sphericalharmonic} 
\scriptq(F_1,F_2,F_3)
= \sum_{n=1}^\infty 
\scriptq(\pi_n(F_1),\pi_n(F_2),\pi_n(F_3)).
\end{equation}

The compactness of the linear operators $T_k$ has the following consequence.
\begin{lemma} \label{lemma:compactnessconsequence}
For each $d\ge 2$
there exists a sequence $\Lambda_n=\Lambda_n(\br)$ satisfying $\lim_{n\to\infty} \Lambda_n=0$
such that for each $n\ge 0$ and all $(G_1,G_2,G_3)\in\scripth_n^3$,
\begin{equation} |\scriptq(G_1,G_2,G_3)| \le \Lambda_n \sum_{j=1}^3 \norm{G_j}_{L^2(S^{d-1})}^2.
\end{equation} \end{lemma}

Denote by $I$ the identity mapping $I:\reals^d\to\reals^d$.
The next two results will be proved in \S\ref{section:centralize}, below.
\begin{lemma} \label{lemma:centering}
Let $d\ge 2$.
Let $\bE$ be as above, and let $\delta = \dist(\bE,\scripto(\bEstar))$.  There exist $\bv,\psi$
satisfying $|\bv|=O(\delta)$ and $\norm{\psi-I}=O(\delta)$
such that the functions $\tilde F_j$ associated
to the sets $\tilde E_j = \psi(E_j)+v_j$ satisfy
\begin{equation} \left\{ \begin{aligned} 
&\pi_1(\tilde F_j)=0 \text{ for $j=1$ and $j=2$}
\\ &\pi_2(\tilde F_1)=0.  \end{aligned} \right. \end{equation}
Here $\bv\in(\reals^d)^3$ satisfies $v_1+v_2+v_3=0$,
and $\psi$ is a Lebesgue measure-preserving linear automorphism of $\reals^d$. 
\end{lemma}
``Associated'', throughout the discussion, means that
$\tilde F_j(\theta) = \int_0^\infty (\one_{E_j}-\one_{B_j})(t\theta)\,t^{d-1}\,dt$.
The norm $\norm{\psi-I}$ is defined by choosing any fixed norm on the vector space
of all $d\times d$ real matrices, expressing the elements $\psi,I$
of the general linear group as such matrices, and taking the norm of the
difference of the two associated matrices.

There are no spherical harmonics of degrees $>1$ for $d=1$, and the group of measure-preserving
linear automorphisms is trivial, so the situation is simpler. 
\begin{lemma} \label{lemma:d=1balancing}
Let $d=1$. 
Let $\bE$ be as above, and let $\delta = \dist(\bE,\scripto(\bEstar))$.  There exists $\bv\in\reals^3$
satisfying $|\bv|=O(\delta)$ and $\sum_{j=1}^3 v_j=0$
such that the functions $\tilde F_j$ associated
to the sets $\tilde E_j = E_j+v_j$ vanish identically on $S^0$.
\end{lemma}

The symmetry group of the functional $\scriptt$ is not sufficiently large 
to enable further reductions of these types, as a simple dimension count demonstrates.

For $d=1$, $F_j=\pi_1(F_j)$ for each index $j$, because there are no spherical harmonics
of higher degrees.
Therefore Lemma~\ref{lemma:d=1balancing}
suffices to complete the proof of Proposition~\ref{prop:1},
hence the proof of Theorem~\ref{thm:RSsharpenedscriptt}, for $d=1$.

For $d\ge 2$,
Lemma~\ref{lemma:centering} eliminates crucial terms from \eqref{eq:sphericalharmonic}.
In particular, spherical harmonics of degree $2$ are eliminated;
if two of the three functions $F_m$ vanish then $\scriptq_k(F_i,F_j)=0$ for all $(i,j,k)$.
Some analysis will be required to show that elimination of these terms suffices
to make the optimal constant $A$ strictly less than $\tfrac12$.

The significance of the conclusions 
$|\bv|=O(\delta)$ and $\norm{\phi-I}=O(\delta)$
is that these ensure that the sets $\tilde E_j = \phi(E_j)+v_j$
continue to satisfy
\[ \tilde E_j\symdif E_j^\star \subset\{x: \big|\,|x|-r_j\,\big|\le C\delta\}\]
for a certain finite constant $C$. Therefore the above analysis applies equally well to these
sets, and we can simply replace $\bE$ by $\tilde\bE$ henceforth.


For $d\ge 2$, replace $E_j$ by $\tilde E_j$ for all three indices.
In order to prove Proposition~\ref{prop:1}
and hence Theorem~\ref{thm:RSsharpenedscriptt} for arbitrary dimensions, 
it now suffices to prove the following two results.

\begin{lemma}\label{lemma:nge3}
Let $d\ge 2$.
Let $(r_1,r_2,r_3)$ be strictly admissible.
For each $n\ge 3$ there exists
$A<\tfrac12$ such that 
for all ordered triples $(G_1,G_2,G_3)$ of spherical harmonics $G_j:S^{d-1}\to\reals$ of degree $n$,
\begin{equation}
\scriptq(G_1,G_2,G_3) \le A \sum_{k=1}^3 \gamma_k r_k^{1-d} \norm{G_k}_{\lt(S^{d-1})}^2.
\end{equation}
\end{lemma}

\begin{lemma}\label{lemma:n=2}
Let $d\ge 2$.
Let $(r_1,r_2,r_3)$ be strictly admissible.
There exists $A<\tfrac12$ such that 
for all spherical harmonics $G_2,G_3:S^{d-1}\to\reals$ of degree $2$,
\begin{equation}
\scriptq_1(G_2,G_3)
\le A \sum_{k=2}^3 \gamma_k r_k^{1-d} \norm{G_k}_{\lt(S^{d-1})}^2.
\end{equation}
\end{lemma}

These two results will be proved in \S\ref{section:finalsteps}.
Together with Lemma~\ref{lemma:compactnessconsequence}, Lemma~\ref{lemma:n=2} gives this corollary:

\begin{corollary}
For each $d\ge 2$ there exists $A<\tfrac12$ with the following property. Let $F_j\in \lt(S^{d-1})$.
Suppose that $\pi_0(F_j)=0$ for all $j\in\{1,2,3\}$,
that $\pi_1(F_j)=0$ for $j=1,2$,
and that $\pi_2(F_1)=0$.
Then
\begin{equation} \label{eq:Abetter}
\scriptq(F_1,F_2,F_3) \le A\sum_{k=1}^3 \gamma_k r_k^{1-d} \norm{F_k}_{\lt}^2.
\end{equation}
\end{corollary}

Rather than calculating $\gamma_j$ and the eigenvalues of the operators associated
to the quadratic forms $Q_j$ on $\scripth_n$ (all of which are functions of
$\br = (r_1,r_2,r_3)$), we will carry out a more conceptual analysis of the difference 
$ \scriptq(G_1,G_2,G_3)
-\tfrac12 \sum_{k=1}^3 \gamma_k r_k^{1-d} \norm{G_k}_{\lt(S^{d-1})}^2$
for ordered triples $(G_1,G_2,G_3)$ of spherical harmonics of degree $n$.

\section{Interlude}
We digress to explain why we are not able to analyze $\scriptq_k$
by means of explicit formulas for spherical harmonics. 
$\scriptq_k(G_i,G_j)$ takes the form
$\langle S_\rho(G_i),\,G_j\rangle$ where the inner product
is that of $\lt(S^{d-1})$,
and $S_\rho$ is the linear operator on $\lt(S^{d-1})$ defined by
\[S_\rho G(x) = \int_{S^{d-1}} G(y)\one_{x\cdot y\le \rho}\,d\sigma(y)\]
where $\rho = (2r_ir_j)^{-1}(r_k^2-r_i^2-r_j^2)$.
All $\rho$ in a certain open interval arise
from admissible ordered triples $(r_1,r_2,r_3)$.

Acting on spherical harmonics of degree $k$, 
$S_\rho$ is a scalar multiple $\lambda_k(\rho)$ of the identity.
Let $P_k$ be the Gegenbauer polynomials. 
These can be defined by the generating function expansion
\[ (1+s^2-2st)^{-(d-2)/2} = \sum_{k=0}^\infty P_k(t)s^k.\]
Then $Z_k(x) = P_k(x_d)$ is a spherical harmonic of degree $k$;
these are the zonal harmonics, up to scalar factors which are of no consequence here.

The value of $S_\rho(Z_k)$ at the point $N=(0,0,\dots,0,1)$ 
is the integral of $Z_k$ over a spherical cap centered at $N$,
whose radius varies with $\rho$.
Thus a calculation of $\lambda_k(\rho)$ for all $\rho$ 
equivalent to a calculation of the ratio of $S_\rho(Z_k)(N)$ to $Z_k(N)$.
This amounts to a calculation of the indefinite integral $\int P_k(t)(1-t^2)^{(d-3)/2}\,dt$.
An explicit formula for the indefinite integral would give
an explicit formula for $P_k(t)$, after differentiation
and division by $(1-t^2)^{(d-3)/2}$.

\section{Balancing via affine automorphisms} \label{section:centralize}

Let $\bb$ be the closed ball of radius $1$, centered at the origin, in $\reals^d$.
Consider bounded Lebesgue measurable sets $E\subset\reals^d$ that satisfy $|E|=|\bb|$.
To $E$ is associated the function $F=F_E:S^{d-1}\to\reals$ defined by 
\begin{equation}\label{eq:FassocE} F_E(\theta) =\int_0^\infty (\one_E-\one_\bb)(t\theta)\,t^{d-1}\,dt.
\end{equation}

\begin{definition}
Let $D\in\naturals$.
A bounded Lebesgue measurable set $E\subset\reals^d$
satisfying $|E|=|\bb|$ is balanced up to degree $D$ if
the function $F_E:S^{d-1}\to\reals$ associated to $E$ by \eqref{eq:FassocE} satisfies
\begin{equation}\label{eq:balanceintegral} \int_{S^{d-1}}F_E(y)P(y)\,d\sigma(y)=0 \end{equation}
for every polynomial $P:\reals^d\to\reals$ of degree less than or equal to $D$.
\end{definition}
For $D=0$, \eqref{eq:balanceintegral} asserts
that $\int (\one_E-\one_\bb)=0$, which is simply a restatement of the hypothesis $|E|=|\bb|$.

Denote by $\aff(d)$ the group of all affine automorphisms of $\reals^d$.
Denote by $\scriptm_d$ the vector space of all $d\times d$ square matrices with real entries, 
and by $\scriptm_d\oplus\reals^d$ the set of all ordered pairs $(T,v)$ where
$T\in \scriptm_d$ and $v\in\reals^d$, with the natural vector space structure.
Identify elements of $\scriptm_d$ with linear endomorphisms of $\reals^d$ in the usual way;
$S=(T,v)$ acts by $S(x) = T(x)+v$.
Fix any norm $\norm{\cdot}_{\scriptm_d}$ on $\scriptm_d$.

Elements $\phi\in\aff(d)$ take the form $\phi(x) = T(x)+v$ where 
$T:\reals^d\to\reals^d$ is an invertible linear transformation, and $v\in\reals^d$.
$T$ can be identified with an element of $\scriptm_d$, and $(T,v)$ is thus identified
with a unique element of $\scriptm_d\oplus\reals^d$. 
Define $\norm{\phi}_{\aff(d)}= \norm{T}_{\scriptm_d} + \norm{v}_{\reals^d}$.
We abuse notation by writing $\det(\phi)$ for the determinant of the unique $T\in\scriptm_d$
thus associated to $\phi$, and likewise $\trace(\phi) = \trace(T)$.

\begin{lemma}\label{lemma:balanced}
Let $d\ge 1$. There exists $c>0$ 
such that for every $\lambda\ge 1$
there exists $C_\lambda<\infty$ with the following property.
For any Lebesgue measurable set $E\subset\reals^d$ satisfying $|E|=|\bb|$,
$\lambda\defb\le c$, and 
$E\symdif\bb\subset\{x: \big|\, |x|-1\,\big|\le \lambda\defb\}$,
there exists a measure-preserving affine transformation $\phi\in\aff(d)$ such that
\begin{gather*}
\text{$\phi(E)$ is balanced up to degree $2$,} 
\\ \norm{\phi-I}_{\aff(d)} \le C_\lambda \defb,
\\ \phi(E)\symdif\bb\subset \{x: \big|\,1-|x|\,\big| \le C_\lambda \defb\}.
\label{eq:nearness2} 
\end{gather*}
\end{lemma}

Denote by $W_2$ the real vector space of all polynomials $P:\reals^d\to\reals$
that are finite linear combinations of homogeneous harmonic polynomials of degrees $\le 2$.
Denote by $V_2$ the real vector space
of all restrictions to $S^{d-1}$ of real-valued polynomials of degrees $\le 2$.
The natural linear mapping from $W_2$ to $V_2$ induced by restriction from
$\reals^d$ to $S^{d-1}$ is a bijection \cite{steinweiss}.

Regard $V_2$ as a real inner product space, with the $L^2(S^{d-1},\sigma)$ inner product.
Denote by $\Pi$ the orthogonal projection of $L^2(S^{d-1})$ onto its subspace $V_2$.
Define $\aA: \scriptm_d\oplus\reals^d \to V_2$ by
\begin{equation}
\aA(S)(\alpha) = \Pi(\alpha\cdot S(\alpha)),
\end{equation}
that is, the right-hand side equals the
restriction to $S^{d-1}$ of the quadratic polynomial $\reals^d\owns x\mapsto x\cdot S(x)$.

\begin{lemma} $\aA: \scriptm_d\oplus\reals^d\to V_2$ is surjective.  \end{lemma}

\begin{proof}
The range of $\aA$ is
the collection of all functions $S^{d-1}\owns \alpha\mapsto S(\alpha)\cdot\alpha$,
as the function $S$ varies over all affine mappings
from $\reals^d$ to $\reals^d$. Because $S\mapsto\aA(S)$ is linear,
this range is a subspace of $V_2$.

Firstly, the constant function $\alpha\mapsto 1$ equals $\aA(S)$ when $S(x)\equiv x$,
since $S(\alpha)\cdot\alpha = \alpha\cdot\alpha\equiv 1$ for $\alpha\in S^{d-1}$.
Secondly, a linear monomial $\alpha=(\alpha_1,\dots,\alpha_d)\mapsto\alpha_k$ is expressed
by choosing $S(x)\equiv e_k$, the $k$--th unit coordinate vector.
Thirdly, to express a monomial $\alpha\mapsto\alpha_j\alpha_k$ in the form $S(\alpha)\cdot\alpha$,
define $S(x)=(S_1(x),\dots,S_d(x))$ by $S_i(x)\equiv 0$ for all $i\ne j$, and $S_j(x)=x_k$.
Then $\alpha_j\alpha_k = S(\alpha)\cdot\alpha$. 
Functions of these three types span $V_2$, so $\aA$ is indeed surjective.
\end{proof}

\begin{proof}[Proof of Lemma~\ref{lemma:balanced}]
If $c\le\tfrac12$ then $E$ contains the ball of radius $\tfrac12 r$ centered at $0$,
so if $\phi\in\aff(d)$ is sufficiently close to the identity then $\phi(E)$ contains
the ball of radius $\tfrac14 r$ centered at $0$.

Let $k\in\{0,1,2\}$.
Let $P:\reals^d\to\reals$ be a homogeneous harmonic polynomial of degree $k$.
Let $g(x)$ be a smooth function that agrees with $|x|^{-k}P(x)$
in $\{x: \big|\,|x|-1\,\big|\le\tfrac34\}$. 
Moreover, choose $g$ so that the map $P\mapsto g$ is linear over $\reals$.

For $\phi\in\aff(d)$ let $f_{\phi(E)}=\one_{\phi(E)}-\one_\bb$ and $F_{\phi(E)}$ 
be the functions associated to $\phi(E)$
in the same way that $f=\one_E-\one_\bb$ and $F$ are associated to $E$.
Then if $\phi$ is sufficiently close to the identity,
\begin{align}
\int_{S^{d-1}} F_{\phi(E)}(y)P(y)\,d\sigma(y)
&= \int_{S^{d-1}}\int_0^\infty (\one_{\phi(E)}-\one_\bb)(ry)\,r^{d-1}\,dr\,P(y) d\sigma(y)
\notag
\\ &= \int_{\reals^d}  (\one_{\phi(E)}-\one_\bb)(x)|x|^{-k} P(x) \,dx
\notag
\\ &= \int_{\reals^d}  (\one_{E}\circ\phi^{-1}-\one_\bb)\, g
\notag
\\ &= \int_{\reals^d}  (f\circ\phi^{-1})\, g
+  \int_{\reals^d}  (\one_\bb\circ\phi^{-1}-\one_\bb)\, g.
\label{eq:cocycle?}
\end{align}
The second to last equation holds because both $\bb$ and $\phi(E)$ contain
the ball of radius $\tfrac14$ centered at $0$, and $g(x)\equiv |x|^{-k}P(x)$
for all $x$ in the complement of this ball.
All of these quantities depend linearly on $P$.

We seek the desired $\phi\in\aff(d)$ in the form $\phi=I+S$,
where $\norm{S}_{\aff(d)}$ is small and $I$ is the identity matrix;
that is, $\phi(x)=x+S(x)$ where $S$ is an affine mapping.
The second term on the right-hand side of \eqref{eq:cocycle?} is independent of $E$.
Moreover,
\begin{align*}
\int_{\reals^d} (\one_\bb\circ\phi^{-1}) \,g
&= |\det(\phi)| \int_{\bb} g\circ\phi 
\\&= (1+\trace(S)) \int_{\bb} g\circ\phi
+ O_P(\norm{S}_{\aff(d)}^2)
\\&= (1+\trace(S)) \int_{\bb} \,g(x+S(x))\,dx 
+ O_P(\norm{S}_{\aff(d)}^2).
\end{align*}
Here and below, 
$O_P(\norm{S}_{\aff(d)}^2)$ denotes a quantity that depends linearly on $P$,
whose norm or absolute value, as appropriate, is bounded above by a constant multiple
of the norm of $P$ times the $\aff(d)$ norm squared of $S$.

Invoking the Taylor expansion of $g$ about $x$ gives
\begin{align*}
\int_{\reals^d} (\one_\bb\circ\phi^{-1}) \,g
&= (1+\trace(S)) \int_{\bb} g
+ \int_{\bb} \, \nabla g\cdot S
+ O_P(\norm{S}_{\aff(d)}^2)
\\&=
\int_{\bb} g + \int_{\bb} \big(g\trace(S) + \nabla g\cdot S\big)
+ O_P(\norm{S}_{\aff(d)}^2)
\\&= \int_{\bb} g + \int_{\bb} \diver(gS) + O_P(\norm{S}_{\aff(d)}^2)
\\& = \int_{\bb} g 
+ \int_{S^{d-1}} g(\alpha) S(\alpha) \cdot \alpha \,d\sigma(\alpha)
+ O_P(\norm{S}_{\aff(d)}^2)
\\& = \int_{\bb} g 
+ \int_{S^{d-1}} P(\alpha) S(\alpha) \cdot \alpha \,d\sigma(\alpha)
+ O_P(\norm{S}_{\aff(d)}^2).
\end{align*}
The second to last equality is justified by the divergence theorem,
and the last by the identity $g\equiv P$ on $S^{d-1}$.  Thus
\begin{equation*}
\int_{\reals^d}  (\one_\bb\circ\phi^{-1}-\one_\bb)\, g
= \int_{S^{d-1}} P(\alpha) S(\alpha) \cdot \alpha \,d\sigma(\alpha)
+ O_P(\norm{S}_{\aff(d)}^2).
\end{equation*}

Since $\int_{\reals^d} (f\circ\phi^{-1})\,g =\int_{\reals^d} (f\circ\phi^{-1})\, P(x)|x|^{-k}$,
by returning to \eqref{eq:cocycle?} we find that the equation 
$\int_{S^{d-1}} F_{\phi(E)}P\,d\sigma=0$ for all $P\in V_2$, for an unknown $S\in\aff(d)$, 
takes the form
\begin{equation} \int_{S^{d-1}} P(\alpha) S(\alpha) \cdot \alpha \,d\sigma(\alpha)
= - \int_{\reals^d} (f\circ\phi^{-1})\,P(x)|x|^{-k}\,dx
+ O_P(\norm{S}_{\aff(d)}^2)\ \forall\,P\in V_2.  \end{equation}
All three terms in this equation depend linearly on $P$, so by interpreting
each term as the inner product of $P$ with an element of $V_2^*$ we may regard
this as an equation in $V_2^*$, thus eliminating $P$.
Equivalently, it will be regarded as an equation in the Hilbert space $V_2$. 


Write this equation as
\begin{equation} \aA(S)=\scriptn_f + \scriptr(S)\end{equation}
where $\aA$ is defined above, 
$\scriptn_f$ is the mapping $P\mapsto  -\int_{\reals^d} (f\circ\phi^{-1})\,P(x)|x|^{-k}\,dx$,
and $\scriptr$ represents the term $P\mapsto O_P(\norm{S}_{\aff(d)}^2)$. 
Both $\aA$ and $\scriptr$ are twice continuously differentiable
functions of $S\in\scriptm_d\oplus\reals^d$.
Moreover,
\begin{equation} \norm{\scriptn_f}(S)_{V_2^*}\le C\defb\end{equation}
simply because $|f|\le \one_{\defb}$ and $f$ is supported where $\tfrac12\le |x|\le \tfrac32$.
Since $\aA:\scriptm_d\oplus\reals^d \to V_2$ is surjective, the Implicit Function Theorem
guarantees that the equation
$\aA(S) =\scriptn_f +\scriptr(S)$
admits a solution $S\in\scriptm_d\oplus\reals^d$
satisfying $\|S\|_{\scriptm_d\oplus\reals^d} \le C|E\symdif \bb|$.
\end{proof}

\begin{lemma}\label{lemma:degree1balanced}
Let $d\ge 1$. There exists $c>0$ 
such that for every positive constant $\lambda < \infty$
there exists $C_\lambda<\infty$ with the following property.
For any Lebesgue measurable set $E\subset\reals^d$ satisfying $|E|=|\bb|$,
$\lambda\defb\le c$, and 
$E\symdif\bb\subset\{x: \big|\, |x|-1\,\big|\le \lambda\defb\}$,
there exists $v\in\reals^d$ such that
\begin{gather*}
\text{$E+v$ is balanced up to degree $1$,} 
\\ |v| \le C_\lambda \defb,
\\ (E+v) \symdif\bb\subset \{x: \big|\,1-|x|\,\big| \le C_\lambda \defb\}.
\label{eq:nearness1} 
\end{gather*}
\end{lemma}
Again, the constant $C_\lambda$ depends on the constant $\lambda$ and on the dimension $d$, but not on 
the set $E$.

The proof of Lemma~\ref{lemma:degree1balanced} is a simplified variant of the proof
of Lemma~\ref{lemma:balanced}. No additional ideas are required. Details are omitted.
\qed

Lemma~\ref{lemma:centering} is a direct application of Lemmas~\ref{lemma:balanced} and \ref{lemma:degree1balanced}
together with dilation.
Choose $\psi$ and $v_1$ so that $\tilde E_1 = \psi(E_1)+v_1$ satisfies the desired conclusion;
$\pi_n(\tilde F_1)=0$ for $n=1$ and $n=2$. Define $v_2=0$ and $v_3=-v_1$,
and define $\tilde E_j = \psi(E_j)+v_j$ for $j=2,3$.
Rename these new sets to be $E_j$, and begin again.
Now choose $\psi$ to be the identity, and define $\tilde E_2=E_2+v_2$
where the new vector $v_2$ is chosen so that $\pi_1(\tilde E_2)=0$.
Define new vectors $v_1,v_3$ by $v_1=0$ and $v_3=-v_2$, 
and define $\tilde E_j = E_j + v_j$ for $j=1,3$.
The resulting doubly modified ordered triple of sets satisfies the conclusions
of Lemma~\ref{lemma:centering}.

Likewise, Lemma~\ref{lemma:d=1balancing} follows directly from Lemma~\ref{lemma:degree1balanced}.
\qed

\section{Final steps} \label{section:finalsteps}

Let $n\ge 3$, and let $\bG = (G_1,G_2,G_3)$ be an ordered triple of spherical harmonics of common degree $n$
on $S^{d-1}$, 
satisfying $\norm{\bG}^2 = \sum_{j=1}^3 \norm{G_j}_{\lt}^2=1$. 
For each $j\in\{1,2,3\}$ define $\varphi_j:S^{d-1}\times (-\tfrac12,\tfrac12) \to\reals^+$ as follows.
If $sG_j(\theta)\ge 0$ then $\varphi_j(\theta,s)\ge 0$,
and $\int_{r_j}^{r_j+\varphi_j(\theta,s)} t^{d-1}\,dt = sG_j(\theta)$.
If $sG_j(\theta) \le 0$ then $\varphi_j(\theta,s)\le 0$,
and $\int_{r_j+\varphi_j(\theta,s)}^{r_j} t^{d-1}\,dt = -sG_j(\theta)$.
Equivalently, for either sign, 
\begin{equation} (r_j+\varphi_j(\theta,s))^d-r_j^d = dsG_j(\theta).\end{equation}
Thus \begin{equation}\label{eq:varphij} r_j^{d-1}\varphi_j(\theta,s) = sG_j(\theta)+O(s^2),\end{equation}
and in a neighborhood of $s=0$, $(\theta,s)\mapsto \varphi_j(\theta,s)$
is a $C^\infty$ function specified by \eqref{eq:varphij}.

For $s\in\reals$ with $|s|$ small define sets $E_j(s)\subset\reals^d$ by
\begin{equation} E_j(s) = \{t\theta: t \le r_j + \varphi_j(\theta,s)\}.\end{equation}
Since $\int_{S^{d-1}} G\,d\sigma=0$,
$|E_j(s)|=|E_j|$ for all $s$ in a neighborhood of $0$.
The function $F_{j,s}$ associated to $E_j(s)$ depends smoothly on $(\theta,s)$
and satisfies $F_{j,s}\equiv sG_j + O(s^2)$.
Define 
\begin{equation} \bE(s) = \bE_{\bG}(s) =  (E_j(s): 1\le j\le 3). \end{equation}

\begin{lemma} \label{lemma:retrace}
For any $d\ge 2$, $\rho>0$, and $n\in\naturals$ 
there exists $c>0$
such that uniformly for each $\rho$--strictly admissible $\br$ satisfying $\max_j r_j=1$
and for all $3$--tuples $\bG$ of spherical harmonics of degree $n$ satisfying $\norm{\bG}=1$,
there exists $\eta>0$ such that
\begin{equation} \label{eq:retrace} \scriptt(\bE(s)) = \scriptt(\bEstar)
-\tfrac12 s^2 \sum_{k=1}^3 \gamma_k r_k^{1-d}   \norm{G_k}_{\lt}^2
+ s^2 \scriptq(\bG) + O(|s|^3)\end{equation}
whenever $|s|\le\eta$.
\end{lemma}

Lemma~\ref{lemma:retrace} 
and Proposition~\ref{prop:RSperturbative} are closely related, but differ in essential ways.
The lemma has the stronger conclusion; 
in the lemma the two sides of the equation
are equal, whereas in the proposition, the left-hand side is less than or equal
to the right-hand side. The stronger conclusion does not hold under the hypotheses
of the proposition. On the other hand, the lemma applies only to a very special
class of sets. Two properties of these sets make possible a more detailed analysis
of the terms $\langle K_k,f_k\rangle$ for $\bE(s)$, which leads to the stronger
conclusion. Firstly,
$F_k^\pm$ have disjoint supports, so that $\norm{F_k}_{\lt}^2 = \norm{F_k^+}_{\lt}^2
+ \norm{F_k^-}_{\lt}^2$.
Secondly, for each $\theta\in S^{d-1}$,
$\{t\in\reals^+: t\theta\in E_k\setminus B_k\}$ is an interval whose left endpoint equals $r_k$,
and likewise
$\{t\in\reals^+: t\theta\in B_k\setminus E_k\}$ is an interval whose right endpoint equals $r_k$.
Combining these two facts with the proof of Proposition~\ref{prop:RSperturbative}
establishes Lemma~\ref{lemma:retrace}.
\qed

The next two lemmas, \ref{lemma:quadgain} and \ref{lemma:quadgain2}, will be proved below.
\begin{lemma}\label{lemma:quadgain} Let $n\ge 3$, $d\ge 2$, and $\rho>0$.
There exists $c>0$, depending on $n,d,\rho$ 
such that for each $\rho$--strictly admissible $\br$ satisfying $\max_j r_j=1$,
for all $3$--tuples $\bG$ of spherical harmonics of degree $n$ satisfying $\norm{\bG}=1$,
\begin{equation}\label{eq:quadgain}
\scriptt(\bE(s))\le \scriptt(\bEstar)-cs^2\end{equation}
for all $s\in\reals$ sufficiently close to $0$.
\end{lemma}
The conclusion holds uniformly for all $s$ in a neighborhood of $0$ that is independent of $\bG$.
This neighborhood, and the constant $c$, are permitted to depend on $n,d,\rho$.
What is essential for the application is that $c$ is independent of $s,\bG$
for all $s$ sufficiently close to $0$.

\begin{lemma}\label{lemma:quadgain2} Let $d\ge 2$ and $\rho>0$.
There exists $c>0$, depending on $d,\rho$ 
such that for each $\rho$--strictly admissible $\br$ satisfying $\max_j r_j=1$,
for all $3$--tuples $\bG$ of spherical harmonics of degree $2$ satisfying $\norm{\bG}=1$
and $G_1=0$,
\begin{equation}\label{eq:quadgain2}
\scriptt(\bE(s))\le \scriptt(\bEstar)-cs^2\end{equation}
for all $s\in\reals$ sufficiently close to $0$.
\end{lemma}

\begin{proof}[Proof of Lemma~\ref{lemma:nge3}]
Upon dividing by $s^2$ in \eqref{eq:retrace} and extracting the limit as $s\to 0$, 
recalling the normalization $\norm{\bG}=1$, we conclude from \eqref{eq:quadgain} that
there exists $c'=c'(n,d,\rho)>0$ such that 
for all ordered triples of spherical harmonics of common degree $n$,
\begin{equation} -\tfrac12 \sum_{k=1}^3 \gamma_k r_k^{1-d}   \norm{G_j}_{\lt}^2
+  \scriptq(\bG) \le -c'\norm{\bG}^2.\end{equation}
As noted in the above discussion of the lack of need for bounds uniform in $n\ge 3$,
this is equivalent to the conclusion of Lemma~\ref{lemma:nge3}.
\end{proof}

In the same way, Lemma~\ref{lemma:n=2} is a direct consequence of Lemma~\ref{lemma:quadgain2}.  \qed

The remainder of this section is devoted to the proofs of Lemmas~\ref{lemma:quadgain}
and \ref{lemma:quadgain2}.
To begin the proof of Lemma~\ref{lemma:quadgain},
let $n,d,\rho$, $\br$, $\bG$ be given.
For any fixed degree $n$, the hypothesis $\norm{\bG}=1$
implies upper bounds on each $G_j$ in $C^\infty(S^{d-1})$. 
For each $s\in\reals$ with small absolute value define $\bE(s)$ as above.
We now proceed to analyze $\scriptt(\bE(s))$ directly, 
without using the reduction to $S^{d-1}$ developed earlier in the analysis.
Via \eqref{eq:retrace}, this will give the desired control on the optimal
constant $A$ in \eqref{eq:A}.

Define $\Sigma$ to be the set of all $\bx'=(x'_1,x'_2,x'_3)\in (\reals^{d-1})^3$
that satisfy $x'_1+x'_2+x'_3=0$.
For each index $j$, define
\[I_j(x',s) = \{t\in\reals: (x',t)\in E_j(s)\}.\]
From the uniform upper bounds for $\bG$  and all of its derivatives,
and from the $\rho$--strict admissibility hypothesis, it follows that there exist 
a neighborhood $V$ of $(0,0,0)\in (\reals^{d-1})^3$
and $\eta>0$
such that for all $\bx'\in V\cap\Sigma$ and all $s\in[-\eta,\eta]$,
each set $I_j(x'_j,s)\subset\reals^1$ is an interval, and 
$(|I_j(x'_j,s)|: 1\le j\le 3)$ is a $2\rho$--strictly admissible ordered triple
of positive real numbers close to $(r_1,r_2,r_3)$.

Let $c_j(x'_j,s)$ be the center of the interval $I_j(x'_j,s)$.
For $x'\in\reals^{d-1}$ in a small neighborhood of $0$ and for $|s|$ small,
the upper endpoint, $t_+$, of $I_j(x',s)$ is the unique solution $t$ of
\[ |x'|^2 + t^2 = (r_j + \varphi_j(\theta,s))^2\]
where $S^{d-1}\owns \theta = (|x'|^2+t^2)^{-1/2} (x',t)$.
Write $t_0=t_0(x')$ for the positive solution of $|x'|^2+t_0^2=r_j^2$.
Thus by \eqref{eq:varphij},
\[ t_+^2 = r_j^2-|x'|^2 +2r_j^{2-d}sG_j(\theta) + O(s^2)
= t_0^2 +2r_j^{2-d}sG_j(\theta) + O(s^2) \]
so
\[ t_+ = 
t_0(1+2t_0^{-2}r_j^{2-d}sG_j(\theta) +O(s^2))^{1/2}
= t_0+s r_j^{2-d}t_0^{-1}G_j(\theta) +O(s^2).  \]
$G_j$ is equal to the restriction to $S^{d-1}$ of a (unique) homogeneous 
harmonic polynomial of degree $n$, also denoted by $G_j$, defined on $\reals^d$.
Writing $G_j(\theta) = (|x'|^2+t_+^2)^{-n/2}G_j(x',t_+)$ and noting that $t_+=t_0+O(s)$ gives
\begin{align*} t_+ &= 
t_0+s r_j^{2-d}t_0^{-1} (|x'|^2+t_+^2)^{-n/2} G_j(x',t_+) +O(s^2)
\\& = t_0+sr_j^{2-d-n} t_0^{-1} G_j(x',t_0) +O(s^2),
\end{align*}
bearing in mind that $t_0$ is a function of $x'$.
In the same way, the lower endpoint, $t_-$, of $I_j(x',s)$ is
\begin{equation*} t_- 
= -t_0 - sr_j^{2-d-n} t_0^{-1} G_j(x',-t_0) +O(s^2).
\end{equation*}
Therefore 
\begin{equation}
c_j(x',s) = \tfrac12 sr_j^{2-d-n}\,  t_0(x')^{-1}\, \big[G_j(x',t_0(x')) -G_j(x',-t_0(x')) \big] + O(s^2).
\end{equation}
Write $G_j = \Gje+\Gjo$ by expanding $G_j(x',x_d)$ (regarded as a function of $(x',x_d)\in\reals^d$)
as a linear combination of monomials in $x=(x',x_d)$ and defining
$\Gje(x',x_d)$ to be the contribution of all monomials having even degrees with respect to $x_d$,
and $\Gjo(x',x_d)$ to be the contribution of all monomials having odd degrees with respect to $x_d$.
Then
\begin{equation}\label{cjformula} c_j(x',s) = sr_j^{2-d-n} x_d^{-1} \Gjo(x',x_d)+ O(s^2).  \end{equation}
The quantity $r_j^{2-d-n}x_d^{-1}\Gjo(x',x_d)$
is a sum of monomials, in each of which $x_d = (r_j^2-|x'|^2)^{1/2}$ is raised to an even power,
because of the factor of $x_d^{-1}$. 
Therefore we may rewrite this last identity in the form
\begin{equation} \label{eq:cjPj}
c_j(x',s) = sP_j(x')+O(s^2)
\end{equation}
where $P_j:\reals^{d-1}\to\reals$ is a polynomial of degree at most $n-1$,
defined by
\begin{equation}\label{eq:associatedP}
P_j(x') = r_j^{2-d-n} x_d^{-1}\Gjo(x',x_d)\ \text{ with } x_d = (r_j^2-|x'|^2)^{1/2}.
\end{equation}
The coefficients of $P_j$ are bounded above, uniformly in all ordered triples
$\bG$ of spherical harmonics of degree $n$ satisfying $\norm{\bG}=1$.

Write
\begin{equation}\label{eq:use_1}
\scriptt(\bE(s))
= \int_{x'_1+x'_2+x'_3=0} \scriptt_1(I_1(x'_1,s),I_2(x'_2,s),I_3(x'_3,s))\,d\lambda(\bx').
\end{equation}
For any $\bx'$,
\begin{equation}\label{eq:noloss} \scriptt_1(I_1(x'_1,s),I_2(x'_2,s),I_3(x'_3,s))
\le \scriptt_1(I_1(x'_1,s)^\star,I_2(x'_2,s)^\star,I_3(x'_3,s)^\star)\end{equation}
by the one-dimensional Riesz-Sobolev inequality. Crucially, there is an improvement in
the case in which the intervals $\scriptt_1(I_j(x'_j,s)$ do not have compatible centers.

\begin{lemma}
For each $\rho>0$ there exists $a_\rho>0$ with the following property.
For $j\in\{1,2,3\}$ let $I_j\subset\reals$ be closed bounded intervals with centers $c_j$. Suppose that
$(|I_j|: 1\le j\le 3)$ is $\rho$--strictly admissible.
Then \begin{equation} \scriptt_1(I_1,I_2,I_3)\le \scriptt_1(I_1^\star,I_2^\star,I_3^\star)-a_\rho
|c_1+c_2+c_3|^2.\end{equation}
\end{lemma}
The proof of this lemma is straightforward, and is omitted. As an alternative,
one could invoke Theorem~\ref{thm:RSsharpened} for $d=1$; but the case of intervals is much simpler
than that of general sets.

Applying this lemma yields
\begin{multline} \scriptt_1(I_1(x'_1,s),I_2(x'_2,s),I_3(x'_3,s))
\\ \le \scriptt_1(I_1(x'_1,s)^\star,I_2(x'_2,s)^\star,I_3(x'_3,s)^\star)
-a |c_1(x'_1,s)+c_2(x'_2,s)+c_3(x'_3,s)|^2 \end{multline}
for a certain  constant $a>0$,
for all $\bx'$ in a sufficiently small neighborhood of the origin in $(\reals^{d-1})^3$,
uniformly for all sufficiently small $s$.
Therefore by the relation \eqref{eq:cjPj} between $c_j(x',s)$, $s$, and $P_j(x')$, 
for all $\bx'\in\Sigma$ sufficiently close to $(0,0,0)$,
\begin{multline}\label{eq:Pgain}
\scriptt_1(I_1(x'_1,s),I_2(x'_2,s),I_3(x'_3,s))
\\
\le \scriptt_1(I_1(x'_1,s)^\star,I_2(x'_2,s)^\star,I_3(x'_3,s)^\star)
-as^2 \Psharp(\bG)(\bx')^2 + O(s^3)
\end{multline}
uniformly in $\bx',s,\bG$ for fixed $d,n,\rho$,
where $\Psharp(\bG)$ is defined on $\Sigma$ by 
\begin{equation}\label{Psharpdefn}
 P^\sharp(\bG)(\bx') = \sum_{j=1}^3 P_j(x'_j),\end{equation} 
with the polynomial $P_j$ defined in terms of $G_j$ as above.

By a polynomial $P$ of degree $D$ with domain $\Sigma$
we mean any function with domain such that $(x_1,x_2)\mapsto P(x_1,x_2,-x_1-x_2)$
is a polynomial of degree $D$.
Introduce any norm on the vector space of all 
polynomials $P:\Sigma \to\reals$ of degrees $\le n-1$.
Combining \eqref{eq:Pgain} with \eqref{eq:use_1} and \eqref{eq:noloss}, we have established the following lemma.
\begin{lemma}
With the above hypotheses and notations,
\begin{equation}
\scriptt(\bE_\bG(s)) \le \scriptt(\bEstar) - cs^2 \norm{\Psharp(\bG)}^2 + O(s^3).
\end{equation}
\end{lemma}

This does not conclude the proof, for there exist $\bG$ of degrees $n\ge 3$
for which $\bG\ne 0$ but $\Psharp(\bG)\equiv 0$.
However, for any $\scripto\in O(d)$, this reasoning can be applied to $\scripto(\bE) = (\scripto(E_j): 1\le j\le 3)$.
Write $\scripto(\bG) = (G_j\circ\scripto: 1\le j\le 3)$.
Since $\scriptt(\scripto(\bE))\equiv \scriptt(\bE)$, it suffices to prove the following result
in order to complete the proof of Lemma~\ref{lemma:quadgain}.

\begin{lemma} \label{lemma:algebra:atlast}
Let $d\ge 2$ and $n\ge 3$. 
If $\bG$ is a nonzero ordered triple of spherical harmonics of degree $n$ then there exists $\scripto\in O(d)$
such that $\Psharp(\scripto(\bG))\ne 0$.
\end{lemma}

It follows immediately from a simple compactness argument
that for each $n,d,\rho$, the infimum over all $\bG$ of
$\max_{\scripto\in O(d)} \norm{\Psharp(\scripto(\bG))}$
is strictly positive, where $\bG$ ranges over the set of all $3$--tuples of spherical harmonics of
common degree $n$ satisfying $\norm{\bG}=1$.

\begin{proof}[Proof of Lemma~\ref{lemma:algebra:atlast}]
If $\sum_{j=1}^3 P_j(x'_j)$ vanishes identically in a neighborhood in $\Sigma$ of $(0,0,0)$
then $P_j(x')$ must be an affine function of $x'\in\reals^{d-1}$ for each index $j$.
Therefore it suffices to show that for any $k\in\{1,2,3\}$ for which $G_k$
does not vanish identically on $\reals^d$,
there exists $\scripto\in O(d)$ such that the polynomial
$x'\mapsto P_k(x')$ associated to $G_k\circ\scripto$ via \eqref{eq:associatedP} fails to be affine.

Fix such an index $k$. 
$G_k$ has degree equal to $n$.
It is well-known that any measurable solutions $\varphi_j$ 
of the functional equation $\sum_{j=1}^3\phi_j(x'_j)\equiv 0$ on $\Sigma\cap (I_1\times I_2\times I_3)$
must be affine functions in a neighborhood of any point of the intersection of $\Sigma$
with the interior of $I_1\times I_2\times I_3)$. Therefore
if the associated polynomial $P_k:\reals^{d-1}\to\reals$ is not affine,  
then the proof is complete. 

Suppose instead that $P_k$ is affine.
By exploiting the identity $|x'|^2+x_d^2=1$ for $(x',x_d)\in S^{d-1}$ to eliminate powers of $x_d$,
one can express $G_k(x',x_d)$, as a function of $(x',x_d)\in S^{d-1}$, 
in the form $p_1(x')+x_d p_2(x')$ where $p_1,p_2$ are uniquely determined polynomials
of degrees $\le n$ and $\le n-1$, respectively.
Now according to \eqref{eq:associatedP}, $P_k(x') = r_k^{2-d-n} p_2(x')$.
Since $P_k$ is affine, this representation can be simplified to
\[G_k(x)=G_k(x',x_d) = p(x') + (x'\cdot v)x_d + bx_d \qquad \text{for $x\in S^{d-1}$} \]
where $p$ is a real-valued polynomial, $v\in\reals^{d-1}$, and $b\in\reals$.
Since $G_k$ has degree equal to $n>2$, $p$ must have degree equal to $n$.

Consider $\tilde G(x',x_d) = G_k(Tx',x_d)= p(Tx') + (Tx'\cdot v)x_d + bx_d$ 
where $T\in O(d-1)$ is chosen so that the coefficient $b$ of $x_1^n$ for $p(Tx')$ is nonzero.
Consider $\tilde G(S(x',x_d))$ where $S$ is a rotation in the $(x_1,x_d)$ plane;
$S$ preserves the coordinates $x_i$ for $2\le i<d$,
and maps $(x_1,x_d)$ to \[(\cos(\alpha)x_1+\sin(\alpha)x_d,\,-\sin(\alpha)x_1+ \cos(\alpha)x_d),\]
where $\alpha\in\reals$ is a free parameter.
Expanding $\tilde G(S(x',x_d))$ in the canonical form $p(x') + x_d q(x')$,
the monomial $x_1^{n-1}x_d$ occurs with coefficient equal to $nb\alpha + O(\alpha^2)$.
Indeed, the term $(Tx'\cdot v)x_d+bx_d$ has degree less than or equal to $2<n$,
and this upper bound is preserved by the rotation $S$. 
Therefore for all sufficiently small nonzero $\alpha$, this coefficient is nonzero.
For any such $\alpha$, the associated polynomial $P(x')$ fails to be affine. 
\end{proof}

While Lemma~\ref{lemma:algebra:atlast} does not hold for spherical harmonics of degree $n=2$, 
there is a satisfactory substitute, which yields Lemma~\ref{lemma:quadgain2}
in the same way that Lemma~\ref{lemma:algebra:atlast} established Lemma~\ref{lemma:quadgain}.
Let $P_j,\Psharp$ continue to be defined by \eqref{eq:associatedP} and by \eqref{Psharpdefn}, respectively. 

\begin{lemma} \label{lemma:algebra:atlast2}
Let $d\ge 2$. 
If $\bG$ is a nonzero ordered triple of spherical harmonics of degree $2$, and
if $G_1\equiv 0$,  then there exists $\scripto\in O(d)$
such that $\Psharp(\scripto(\bG))\ne 0$.
\end{lemma}

\begin{proof}
We follow the reasoning in the proof of Lemma~\ref{lemma:algebra:atlast}.
If $\sum_{j=1}^3 P_j(x'_j)$ vanishes identically in a neighborhood in $\Sigma$ of $(0,0,0)$,
and if $P_1\equiv 0$, then $P_2,P_3$ must be constant functions.
Therefore for $k=2,3$, $G_k$ takes the form $p(x')  + bx_d$ for some constant $b$,
where $p$ is a polynomial of degree $\le 2$. The coefficient $b$ must vanish, for otherwise 
the term $bx_d$ would be a spherical harmonic of degree $1$.
Thus $G_k$ is a function of $x'$ alone.

This can only hold for the composition of $G_k$ with an arbitrary rotation if $G_k\equiv 0$.
\end{proof}

\section{Variant inequality}

Inequality
\eqref{eq:variantRS} and the Riesz-Sobolev inequality
are quite closely related, as will be seen in the proof of Theorem~\ref{thm:variantinverse} below,
but are not merely restatements of one another. 
For $t\ge 0$ define
\begin{equation}
S_t (A,B) = \{x\in\reals^d: \one_{A}*\one_B(x)>t\}.
\end{equation}
Then for any Lebesgue measurable $A,B\subset\reals^d$ with $|A|,|B|<\infty$, 
\begin{equation}
|A|\cdot|B| = \int_{\reals^d} \one_A*\one_B = \int_0^\infty |S_t(A,B)|\,dt
\end{equation}
and
\begin{equation}
\int_{S_\tau(A.B)} \one_A*\one_B
= \tau|S_\tau(A,B)| + \int_\tau^\infty |S_t(A,B)|\,dt.
\end{equation}
Therefore
\begin{align*}
\int_{\reals^d} \min(\one_A*\one_B,\tau) 
& = \int_{\reals^d} \one_A*\one_B 
- \int_{\one_A*\one_B(x)>\tau} \big((\one_A*\one_B)(x)-\tau\big)\,dx
\\& = |A|\cdot|B| -\int_{S_\tau(A,B)} \one_A*\one_B + \tau|S_\tau(A,B)|.
\end{align*}

Thus \eqref{eq:variantRS} can be equivalently restated as
$\Psi(A,B,\tau) \le \Psi(A^\star,B^\star,\tau)$,
where
\begin{equation}
\Psi(A,B,\tau)
= \int_{S_\tau(A,B)} \one_A*\one_B - \tau|S_\tau(A,B)|.
\end{equation}
That is,
\begin{equation}
\int_{S_\tau(A,B)} \one_A*\one_B - \tau|S_\tau(A,B)|
\le \int_{S_\tau(A^\star,B^\star)} \one_{A^\star}*\one_{B^\star} - \tau|S_\tau(A^\star,B^\star)|.
\end{equation}
Compare this with the Riesz-Sobolev inequality, with $E_1=A$, $E_2=B$, and $E_3=S_\tau(A,B)$, 
which states that
\begin{equation}
\int_{S_\tau(A,B)} \one_A*\one_B 
\le \int_{S_\tau(A,B)^\star} \one_{A^\star}*\one_{B^\star}. 
\end{equation}
There are two differences in comparison to the inequality 
$\Psi(A,B,\tau) \le \Psi(A^\star,B^\star,\tau)$: There are no negative terms $-\tau|S_\tau(\cdot,\cdot)|$,
and the domain of integration
$S_\tau(A^\star,B^\star)$ is changed to $[S_\tau(A,B)]^\star$.
If $|S_\tau(A,B)| = |S_\tau(A^\star,B^\star)|$ then $(S_\tau(A,B))^\star = S_\tau(A^\star,B^\star)$
and the two inequalities become direct restatements of one another.
The relation $\int \one_A*\one_B = \int \one_{A^\star}*\one_{B^\star}$
is valid for all sets $A,B$, and
can be rewritten as $\int_0^\infty |S_t(A,B)|\,dt = \int_0^\infty |S_t(A^\star,B^\star)|\,dt$, but
there is no pointwise inequality relating 
the two quantities $|S_t(E_,E_2)|$ and  $|S_t(E_1^\star,E_2^\star)|$, 
in general. 

\begin{lemma} \label{lemma:prelimkprgt}
Let $F_j\in L^1(\reals^+)$
be nonincreasing, nonnegative functions satisfying $\int_y^\infty F_0(x)\,dx \ge \int_y^\infty F_1(x)\,dx$
for all $y\in\reals^+$.  Then for each $\tau\in\reals^+$,
\begin{equation}
\int_0^\infty \min(F_0(x),\tau)\,dx \ge \int_0^\infty \min(F_1(x),\tau)\,dx.
\end{equation}
\end{lemma}

\begin{proof}
Via simple approximation and limiting arguments we can reduce to the case in which $F_0,F_1$ vanish
outside of some bounded interval, belong to $C^1([0,\infty))$, are strictly decreasing with strictly
negative derivatives where they are nonzero, and satisfy $\sup_x F_j(x)>\tau > 0$.
For $t\in[0,1]$ consider $F(x,t) = tF_1(x)+(1-t)F_0(x)$. It suffices to show that
$\int_0^\infty \min(F(x,t),\tau) \,dx$ is a nonincreasing function of $t$.

Our hypotheses on $F_0,F_1$ guarantee that
for each $t$ there exists a unique $a(t)\in\reals^+$ satisfying $F(a(t),t)=\tau$, 
and that $a$ is a differentiable function of $t$.
Then \[\int_0^\infty \min(F(x,t),\tau)\,dx = 
\int_0^{a(t)} \tau\,dx + \int_{a(t)}^\infty F(x,t)\,dx\] 
and consequently
\begin{align*}
\frac{d}{dt} \int_0^\infty \min(F(x,t),\tau)\,dx
& = \tau a'(t) -\tau a'(t) + \int_{a(t)}^\infty \frac{\partial F(x,t)}{\partial t}\,dx
\\& =  \int_{a(t)}^\infty (F_1(x)-F_0(x)) \,dx
\\& \le 0.
\end{align*}
\end{proof}

\begin{proof}[Proof of Theorem~\ref{thm:RSvariant}]
Let $F_0:\reals^+\to[0,\infty)$ be right continuous, nonincreasing,
and satisfy $|\{y\in\reals^+: F_0(y)>t\}| = |\{x\in\reals^d: \one_A*\one_B(x)>t\}|$
for all $t\in[0,\infty)$.
Let $F_1$ be associated to $\one_{A^\star}*\one_{B^\star}$ in the same way.
The Riesz-Sobolev inequality states that $\int_0^x F_0\le \int_0^x F_1$ for all $x\in[0,\infty)$.
Since $\int_0^\infty F_0 = |A|\cdot|B| = |A^\star|\cdot|B^\star|=\int_0^\infty F_1$,
this can be equivalently restated as $\int_x^\infty F_0\ge \int_x^\infty F_1$
for all $x\in\reals^+$. Moreover,
$\int_0^\infty \min(F_0,\tau) = \int_{\reals^d} \min(\one_A*\one_B,\tau)$,
with a corresponding identity for $F_1$.
Therefore an application of Lemma~\ref{lemma:prelimkprgt} yields the conclusion of the theorem.
\end{proof}

\begin{lemma}\label{lemma:flowforvariant}
Let $A,B\subset\reals^1$ be Lebesgue measurable sets with finite Lebesgue measures.
Let $A(s),B(s)$ be their flows, as described in Proposition~\ref{prop:RSflow}.
For any $\tau\in\reals^+$,
$\int_{\reals} \min\big(\one_{A(s)}*\one_{B(s)},\,\tau\big)$
is a nonincreasing continuous function of $s\in[0,1]$.
\end{lemma}

\begin{proof}
Continuity is easy, since $t\mapsto \min(t,\tau)$ is a Lipschitz function
and $s\mapsto \one_{A(s)}*\one_{B(s)}$ is a continuous mapping from $[0,1]$ 
to $L^1(\reals^d)$ by Proposition~\ref{prop:RSflow}.

Define $F_s:(0,\infty)\to[0,\infty)$ to be the unique nonincreasing right continuous function
that satisfies \[ |\{x\in\reals^+: F_s(x)>u\}| = |\{y\in\reals: (\one_{A(s)}*\one_{B(s)})(y)>u\}|
\ \text{ for almost every $u\in\reals^+$.}\]
We claim that whenever $s_0\le s_1$, $\int_0^x F_{s_0}(y)\,dy \le \int_0^x F_{s_1}(y)\,dy$.
It suffices to prove this for $s_0=0$. 
Observe that 
for any $s\in[0,1]$ and any $x\in\reals^+$, $\int_0^x F_s(y)\,dy$ is equal to 
the supremum of $\int_E \one_{A(s)}*\one_{B(s)}$,
with the supremum taken over all $E\subset\reals$ satisfying $|E|=x$. 
This supremum is attained.
Choose $E$ so that 
$\int_E \one_{A}*\one_{B} = \int_0^x F_0(y)\,dy$,
and consider the flow $s\mapsto E(s)$ and the associated expression
$\Psi(s) = \int_{E(s)} \one_{A(s)}*\one_{B(s)}$.
According to Proposition~\ref{prop:RSflow}, $\Psi$ is a nondecreasing function.
But
\begin{multline*} \int_0^x F_0(y)\,dy = 
\int_{E(0)} \one_{A(0)}*\one_{B(0)}
\le \int_{E(s)} \one_{A(s)}*\one_{B(s)} 
\\
\le \sup_{|\tilde E|=|E(s)|} \int_{\tilde E} \one_{A(s)}*\one_{B(s)}
= \int_0^x F_s(y)\,dy\end{multline*}
since $|E(s)|\equiv |E|$.

For any $s$, $\int_{\reals^+} F_s = |A(s)|\cdot |B(s)| = |A|\cdot|B|$.
Therefore since
$|A(s_0)|\cdot|B(s_0)| = |A(s_1)|\cdot|B(s_1)|$, 
the inequality $\int_0^x F_{s_0} \le \int_0^x F_{s_1}$ for all $x$,
can be rewritten as $\int_x^\infty F_{s_0} \ge \int_x^\infty F_{s_1}$ for all $x$.
Therefore an application of Lemma~\ref{lemma:prelimkprgt} completes the proof.
\end{proof}

\begin{proof}[Proof of Theorem~\ref{thm:variantinverse}]
By virtue of the continuity and monotonicity of the functional $\int \min(\one_{A(s)}*\one_{B(s)},\,\tau)$
discussed in Lemma~\ref{lemma:flowforvariant}, 
together with its affine invariance, 
we may reduce matters, as in the proof of Theorem~\ref{thm:RSsharpenedscriptt},
to the small perturbation case, in which $\Dist((A,B),\scripto(A^\star,B^\star))$ is much
less than $\max(|A|,|B|)$.
By making a suitable measure-preserving affine change of variables we may
reduce to the case in which \[\max(|A\symdif A^\star|,\,|B\symdif B^\star|)
\le 2\Dist((A,B),\scripto(A^\star,B^\star)).\]

Write
\[\int \min(\one_A*\one_B,\,\tau)
= |A|\cdot|B| +\tau|S_\tau(A,B)| - \int_{S_\tau(A,B)} \one_A*\one_B. \]
There is a corresponding identity for 
$\int \min(\one_{A^\star}*\one_{B^\star},\,\tau)$, 
and $|A|\cdot|B| = |A^\star|\cdot|B^\star|$. 
Let $S = S_\tau(A,B)$ and $S^\sharp = S_\tau(A^\star,B^\star)$.
Therefore we seek to bound 
$\int_{S}\one_A*\one_B - \tau|S|$
by 
$\int_{S^\sharp}\one_{A^\star}*\one_{B^\star} - \tau|S^\sharp|$,
minus a suitable nonnegative term.

One has
\begin{equation}
\int_{S^*} 
\one_{A^\star}*\one_{B^\star}
\le \int_{S^\sharp}
\one_{A^\star}*\one_{B^\star}
- \tau(|S^\sharp|-|S|) 
\end{equation}
in general,
and
\begin{equation}
\int_{S^*} 
\one_{A^\star}*\one_{B^\star}
\le \int_{S^\sharp}
\one_{A^\star}*\one_{B^\star}
- \tau(|S^\sharp|-|S|) 
- c(|S^\sharp|-|S|)^2
\end{equation}
for ordered triples $(A,B,S)$ in the strictly admissible range. 

From the elementary uniform bound
\[ \norm{\,\one_A*\one_B - \one_{A^\star}*\one_{B^\star}\, }_{L^\infty}
\le |A\symdif A^\star|\cdot|B| + |A|\cdot |B\symdif B^\star|\]
and the assumption that $\max(|A\symdif A^\star|,|B\symdif B^\star|)\ll \max(|A|,|B|)$
it follows that \[\big| S_\tau(A,B)\symdif S_\tau(A^\star,B^\star)\big|\ll \max(|A|,|B|).\]
Therefore the ordered triple $(A,B,S_\tau(A,B))$ is $\varrho$--strictly admissible, where
$\varrho>0$ depends only on the parameters in the hypotheses of Theorem~\ref{thm:variantinverse}.

Therefore the Riesz-Sobolev inequality in the form \eqref{eq:RSsharpened2} can be invoked to obtain
\begin{align*}
\int_S \one_A*\one_B - \tau|S|
& \le
\int_{S^*} \one_{A^\star}*\one_{B^\star} -\tau |S|
-c\Dist((A,B),\scripto(A^\star,B^\star))^2
\\ &\le
\int_{S^\sharp} \one_{A^\star}*\one_{B^\star} 
- \tau(|S^\sharp|-|S|) 
- c(|S^\sharp|-|S|)^2
-\tau |S|
\\& \qquad\qquad
-c\Dist((A,B),\scripto(A^\star,B^\star))^2
\\&
=
\int_{S^\sharp} \one_{A^\star}*\one_{B^\star} 
- \tau|S^\sharp|
- c(|S^\sharp|-|S|)^2
-c\Dist((A,B),\scripto(A^\star,B^\star))^2.
\end{align*}
This completes the proof of Theorem~\ref{thm:variantinverse},
as well as of the formally sharper form, Theorem~\ref{thm:variantinverse2}.
\end{proof}

\section{A property of the flow} \label{section:flowtointervals}

Here we prove Proposition~\ref{prop:flowtointervals},
which states that for any $t>0$, $E(t)$ equals a union of intervals,
up to a Lebesgue null set.

The flow of $E$ can be regarded as a flow $(t,x)\mapsto x(t)\in E(t)$
of the points $x\in E$, in the following natural way.
Firstly, define $\phi_E:E\to E^\star$ by 
\[\phi_E(x) = |E\cap(-\infty,x]| - \tfrac12|E|.\]
$\phi_E$ is a nondecreasing function, and
$|\phi_E(E\cap I)| = |E\cap I|$ for every interval $I$.
Secondly, define $\tilde\phi_E: (-\tfrac12|E|,\tfrac12|E|)\to \reals$
by \[\tilde\phi(x)=y\in\reals\] where
$y$ is the smallest element of $\reals$ satisfying
$|E\cap(-\infty,y]|=y+\tfrac12 |E|$.
$\tilde\phi_E$ is a nondecreasing Lebesgue measure-preserving function. 
It is a consequence of the Lebesgue density theorem
that for almost every $x\in E$, the only point $y\in\reals$
satisfying $|E\cap(-\infty,y]| = |E\cap(-\infty,x]|$
is $y=x$ itself. Therefore $\tilde\phi_E(\phi_E(x))=x$ for almost every $x\in E$, 
and $|E\symdif \tilde\phi_E(E^\star)|=0$.


For each $t\in[0,1]$ let $\phi_{E(t)}:E(t)\to E(t)^\star=E^\star$ 
be defined in this way.
Set $\psi_E(t)=
\tilde\phi_{E(t)}\circ\phi_E: E\to E(t)$. This
is a well-defined nondecreasing function, which preserves
Lebesgue measure of Borel sets.
The mapping $E\owns x \mapsto \psi_E(t)(x)$
defines the desired flow on the underlying points of $E$.

The next lemma states that 
if $I$ is a bounded interval, and if $E\cap I$ is sufficiently
dense in $I$, then $\Psi_t(E)$ contains an interval of length
comparable to $I$, for all $t>0$ that are not too small.

\begin{lemma} \label{lemma:compression}
Let $[a,b]\subset[0,\infty)$ be a closed bounded interval of positive length.
Let $E\subset \reals$ be a Lebesgue measurable set satisfying $0<|E|<\infty$.
Let $\delta\in[0,\tfrac12)$.
Suppose that $|E\cap I| \ge (1-\delta)|I|$.
Then for every $T > 2\delta(1-2\delta)^{-1}$, the set 
$\psi_t(E\cap I)$ is an interval.
\end{lemma}


We will prove Lemma~\ref{lemma:compression}
in the special case in which $E$ is a finite union of closed intervals.
In that case, the mapping $t\mapsto x(t)$ is continuous
and is almost everywhere differentiable for almost every $x\in E$.
Then for each $x\in E$, for each $t$ let $c(t)$ be the center of the
largest interval that is contained in $E(t)$, and contains $x(t)$.
Then for almost every $t$, $dx(t)/dt=-c(t)$.

\begin{proof}[Proof of Lemma~\ref{lemma:compression}]
Write $I =  [a^-,a^+]$ 
and $|I| = a^+-a^-$.
Consider $a^-(t),a^+(t)\in E(t)$.
Define $\eta^-(t)$ to be the supremum of all $\eta\ge0$ such that
\[[a^-(t),a^-(t)+2\eta |I|] \subset E(t).\]
Likewise
define $\eta^+(t)$ to be the supremum of all $\eta \ge 0$ such that
\[[a^+(t)-2\eta|I|,a^+(t)]\subset E(t).\]
$\eta^\pm(t)$ are nondecreasing functions of $t \in[0,1]$.

Let $c^\pm(t)$ be the centers of the largest intervals contained in $E(t)$
that contain $a^\pm(t)$, respectively.
Provided that these two intervals are disjoint,
\[ c^-(t)\le a^-(t)+\eta^-(t) \text{ and }  c^+(t)\ge a^+(t)-\eta^+(t).\]
Moreover, for almost every $t$,
$da^\pm(t)/dt$ exists and satisfies
\[\frac{d}{dt} a^\pm(t) = -c^\pm(t).\]

Let $T>0$ and suppose that $\eta^+(T)+\eta^-(T)< \tfrac12$,
and consequently  the two intervals $[a^-(t),a^-(t)+\eta^-(t)|I|]$
and $[a^+(t),a^+(t)-\eta^-(t)|I|$ are disjoint for each $t\in[0,T]$. 

Since $a^+(t)-a^-(t)\ge |E\cap I| \ge (1-\delta)|I|$,
\begin{align*}
\frac{d}{dt}(a^+(t)-a^-(t))
&= c^-(t)-c^+(t)
\\&
\le \big[a^-(t)+ \eta^-(t)|I|\big] - \big[a^+(t)- \eta^+(t)|I|\big]
\\&
\le \big[-(1-\delta)+(\eta^+(t)+\eta^-(t)\big] |I|
\\&
\le (-\tfrac12 +\delta)|I| 
\end{align*}
for almost every $t\in[0,T]$.
Integrating over $t$ and using the initial condition $a^+(0)-a^-(0)=a^+-a^-=|I|$ gives
\[ a^+(T)-a^-(T)
\le \big[1 +(-\tfrac12+\delta)T\big] |I|.\]
Combining this inequality with the constraint
$a^+(T)-a^-(T) \ge  (1-\delta)|I|$ gives
\[ 1-\delta \le 1-T(\tfrac12-\delta),\]
that is,
$T\le 2\delta(1-2\delta)^{-1}$.

Now consider any $\tau$ strictly greater than $2\delta(1-2\delta)^{-1}$.
The hypothesis $\eta^+(\tau)+\eta^-(\tau) < \tfrac12 $ underlying the above reasoning cannot hold,
since the conclusion does not.
Thus $\eta^+(\tau)+\eta^-(\tau)\ge\tfrac12$.
Therefore the interval $[a^-(\tau),a^+(\tau)]$ is contained in $E(\tau)$, up to a Lebesgue null set.
\end{proof}

Proposition~\ref{prop:flowtointervals} is an immediate corollary of the next lemma.

\begin{lemma}
Let $E\subset\reals$ be a Lebesgue measurable set satisfying $0<|E|<\infty$.
Let $\eps>0$.
There exists $t<\eps$ such that $E(t)$ can be expressed as a countable
union of intervals, together with a set of Lebesgue measure less than $\eps$.
\end{lemma}

\begin{proof}
If $E(\tau)$ can be expressed as a union of countably
many intervals together with a set of Lebesgue measure less than $\eps$
for some $\tau>0$, then the same holds for $E(t)$, for every $t>\tau$.

According to the Lebesgue density theorem,
there exist a collection of pairwise disjoint bounded intervals $I_j$
and a subset $E'\subset E$ such that
$|E'|<\eps$,
$E'\cap I_j=\emptyset$ for each index $j$,
and $|E\cap I_j| \ge (1-\eps)|I_j|$ for every $j$.
According to Lemma~\ref{lemma:compression},
$\psi_{E,t}(E\cap I_j)$ is an interval for each index $j$,
for every $t > 2\eps(1-2\eps)^{-1}$. 
Moreover, $|\psi_{E,t}(E')| = |E'|<\eps$.
Therefore $\psi_{E,t}(E\setminus E')$ can be expressed as a countable union of intervals.
\end{proof}

\end{document}